\documentclass[preprint,3p,12pt]{elsarticle}

\usepackage{amssymb}
\usepackage{latexsym}

\usepackage{url}
\usepackage[dvipsnames]{xcolor}
\definecolor{newcolor}{rgb}{.8,.349,.1}
\usepackage{epic}
\usepackage[normalem]{ulem}

\usepackage{bm}
\usepackage[ruled,vlined]{algorithm2e}
\usepackage[caption = false]{subfig}
\usepackage{amsmath,amsfonts,amsthm}
\usepackage{nccmath}
\usepackage{hyperref}
\hypersetup{colorlinks=true}
\usepackage{capt-of}
\usepackage{array}
\usepackage{stackengine}

\newcommand{\real}{\mathbb{R}}

\newcommand{\dv}[2]{\frac{\partial #1}{\partial #2}}
\newcommand{\ip}[1]{\langle #1 \rangle}
\newcommand{\ipp}[2]{\langle #1, #2 \rangle}

\renewcommand{\boldsymbol}[1]{\bm{#1}} 
\newcommand{\mat}[1]{\mathbf{#1}}

\newcommand{\grad}{\boldsymbol{\nabla}}

\newtheorem{theorem}{Theorem}
\newtheorem{lemma}{Lemma}

\graphicspath{{./figs/}}

\begin{document}


\begin{frontmatter}

\title{
A Reduced-Order Model for Nonlinear Radiative Transfer Problems
Based on Moment Equations and POD-Petrov-Galerkin Projection of the Normalized Boltzmann Transport Equation
}

\author[1,2,3]{Joseph M. Coale}

\author[2,4]{Dmitriy Y. Anistratov}

\address[1]{Computational Physics and Methods Group, Los Alamos National Laboratory, Los Alamos, NM}
\address[2]{Department of Nuclear Engineering, North Carolina State University, Raleigh, NC}
\address[3]{jmcoale@lanl.gov}
\address[4]{anistratov@ncsu.edu}


\begin{abstract}
A data-driven projection-based reduced-order model (ROM) for nonlinear thermal radiative transfer (TRT) problems is presented.
The TRT ROM is formulated  by (i) a hierarchy of low-order quasidiffusion  (aka variable Eddington factor) equations for moments of the radiation intensity and (ii) the normalized Boltzmann transport equation (BTE).
The multilevel system of moment equations is derived by projection of the BTE  onto a sequence of subspaces which represent elements of the phase space of the problem.
Exact closure for the moment equations is provided by the Eddington tensor.
A Petrov-Galerkin (PG) projection of the normalized BTE is formulated using a proper orthogonal decomposition (POD) basis
 representing
the normalized radiation intensity over the whole phase space and time.
The Eddington tensor  linearly depends on the  solution of the normalized BTE.
By linear superposition of the POD basis functions, a low-rank expansion of the Eddington tensor is constructed with coefficients defined by the PG  projected normalized BTE.
The material energy balance  (MEB)  equation is coupled with the effective grey low-order equations  which exist on the same dimensional scale as the MEB equation.
The resulting TRT ROM is structure and asymptotic preserving.
A detailed analysis of the ROM is performed on the classical Fleck-Cummings (F-C) TRT multigroup test problem in 2D geometry.
Numerical results are presented to demonstrate the ROM's effectiveness in the simulation of radiation wave phenomena.
The ROM is shown to produce solutions with sufficiently high accuracy while using low-rank approximation of the normalized BTE solution.
Essential physical characteristics of supersonic radiation wave are preserved in the ROM solutions.

\end{abstract}

\begin{keyword}
Boltzmann transport equation,
high-energy density physics,
nonlinear PDEs,
variable Eddington factor,
proper orthogonal decomposition,
POD-Petrov-Galerkin projection
\end{keyword}

\end{frontmatter}


%
%
\section{Introduction}
%
%
In this paper we formulate and analyze a data-based reduced-order model (ROM) aimed at solving
problems of thermal radiative transfer (TRT). TRT problems give a fundamental description of complex phenomena whose primary mode of energy redistribution is radiative transfer (RT).
Phenomena with significant RT effects can be found in a wide spectrum of different fields including high-energy-density physics, plasma physics, astrophysics, atmospheric and ocean sciences, and fire and combustion physics \cite{zel-1966,mihalas-FRH-1984,shu-astro,thomas-stamnes-atm,faghri-sunden-2008,drake-hedp}.

The models that describe these phenomena are typically made up of complex multiphysical systems of partial differential equations (e.g. radiation hydrodynamic problems). Included in these models is the Boltzmann transport equation (BTE) which describes the propagation of photon radiation through matter. Some of the essential challenges associated with the numerical simulation of these models includes their (i) strong nonlinear effects, (ii) tight coupling between equations, (iii) multiscale characterization in space-time and (iv) high-dimensionality. Note the fundamental TRT problem which forms a basic model of the phenomena in question shares all of these challenges.

The issue of high-dimensionality in particular can be one of the largest computational challenges
to overcome, and in many cases is attributable to the included BTE.
The BTE usually resides on a higher-dimensional space than the other multiphysical equations it becomes coupled to. The BTE solution is a function of 7 independent variables in 3D geometry describing dependence on spatial position, time, photon frequency (photon energy)
and direction of photon motion. Therefore to reasonably resolve the BTE solution in phase space and time on a discrete grid, a massive number of degrees of freedom (DoF) must be used. As an example, consider resolving the BTE solution with just 100 discrete points over each independent variable. This results in $10^{14}$ DoF that must be used in the numerical simulation. The high-dimensionality imposed by the BTE can thus place an enormous computational burden and memory requirement on the simulation. It is for this reason that ROMs for the BTE are typically used in massive multiphysical simulations, since reducing the dimensionality of the BTE will in turn reduce the dimensionality of the entire multiphysics model.

%
%

There exists a class of
ROMs for the BTE  based on moment equations that are well understood and commonly  used in practice. Some of the most widely used are diffusion-based ROMs including flux-limited diffusion, $P_1$ and $P_{1/3}$ models \cite{olson-auer-hall-2000,morel-2000,simmons-mihalas-2000}. There is also the class of variable Eddington factor (VEF)
ROMs which use approximations of the Eddington tensor defined by the first two moments of the BTE solution
\cite{pomraning-1969,levermore-1984,korner-janka-a&a-1992,su-fried-larsen-ttsp-2001}.
$M_N$ type models use a maximum-entropy closure for a system of $N$ moment equations \cite{levermore-1996,hauck-2011,hauck-2012}.
Note that the Minerbo model is $M_1$ \cite{minerbo-1978}.
Other commonly used models apply Kershaw, Wilson, Livermore closures
\cite{levermore-1984,kershaw-1976,wilson-1970,m1-2017}.
All of these ROMs are able to effectively reduce the dimensionality of problems involving some particle transport component and thus also reduce the associated computational requirements to find their solution. As such they are useful in many applications where
available computational power is the limiting factor in numerical simulation.
The advantage of such ROMs is that closures for  the moment equations
do not require a priori knowledge of  problem and its solution.
The essential issue with  these ROMs stems from the limitations to their solution accuracy
resulting from approximations  applied to formulate closures.

In order to generate high-fidelity solutions to problems involving the BTE while maintaining low computational requirements, new and advanced types of ROMs
must be formulated.
One of the most promising avenues towards the development of such advanced classes of ROMs is through the use of data-driven methodologies.
Data-driven ROMs can be created by collecting databases of solutions to a target class of problems and then using a data-driven technique of dimensionality reduction to formulate a problem-specific, low-dimensional description of the original full-order model (FOM) \cite{lucia-beran-silva-2004,hastie-tibshirani-friedman-2009,benner-gugercin-willcox-2015,brunton-kutz-2019,mor-vol2-2020}.
Since these models are created using known data about their target class of problems, they can produce highly accurate solutions.
Some of the available methods for formulating data-based approximations and projections include (i) the proper orthogonal decomposition (POD) (a.k.a. principle component analysis (PCA) or the Karhunen-Lo\`eve expansion) \cite{lumley-1967,sirovich-1987,sirovich-1989,aubry-1991,berkooz-holmes-lumley-1993,holmes-1996,volkwein-1999},
(ii) the dynamic mode decomposition (DMD) \cite{rowley-2009,schmid-2010,tu-rowley-2014,williams-kevrekidis-rowley-2015,smith2022variable,li2022dynamic}, (iii) deep neural networks (DNNs) \cite{goodfellow-bengio-courville-2016},
(iv) the proper generalized decomposition (PGD) \cite{chinesta-2011}, (v) balanced truncation \cite{moore-1981} and (vi) reduced basis methods \cite{quarteroni-2015}.

In recent years, many data-driven ROMs have been developed using these techniques for linear problems involving the BTE.
Dimensionality reduction in the angular variable has been formulated using both the POD and reduced-basis methods \cite{buchan-2015,tencer-2017,soucasse-2019,peng-chen-cheng-li-2022,hughes-buchan-2022}.
POD-Petrov-Galerkin projections have been formulated for the steady-state BTE in 1D geometry \cite{behne-ragusa-morel-2019,behne-ragusa-tano-2021}.
A POD-Galerkin projection in both space and time has also been formulated for the BTE \cite{choi-2021}.
The PGD has been used to seperate variables in the BTE solution in several geometries \cite{prince-ragusa-2019-1,Dominesey-2019,prince-ragusa-2019-2}.
A low-rank manifold projection technique has been used in constructing ROMs for time-dependent problems \cite{peng-mcclarren-frank-2019,peng-mcclarren-tans2019,peng-mcclarren-frank-2020,peng-mcclarren-2021}.
Surrogate models for the BTE have been created using the DMD \cite{hardy-morel-ahrens-2019} and various neural network architectures \cite{pozulp-2019,pozulp-brantley-palmer-vujic-2021,elhareef-wu-ma-2021}.
Neural networks have been used to obtain closures for $P_N$-type systems of moment equations of the BTE \cite{huang-I-2022,huang-II-2021}.
Data-driven ROMs have also been created for particle transport problems in nuclear reactor-physics applications
\cite{kord-pne-1986,kord-1995,sanchez-pne-2009,kord-pne-2017}, including (i) pin-by-pin reactor calculations \cite{cherezov-sanchez-joo-2018}, (ii) reactor kinetics \cite{alberti-palmer-2019,alberti-palmer-2020}, (iii) molten salt fast reactor problems \cite{german-ragusa-2019-1,german-ragusa-2019-2,german-tano-ragusa-2021}, (iv) problems with feedback from delayed neutron precursors \cite{elzohery-roberts-2021-1,elzohery-roberts-2021-2,elzohery-roberts-2021-3}, (v) problems with domain decomposition \cite{phillips-2021}, and (vi) for generation of neutron flux and cross sections for light water reactors \cite{Dominesey-2022}.

Furthermore, several data-based ROMs have been developed to date for the nonlinear problems of radiative transfer considered here.
Dimensionality along the frequency variable has been formulated with the POD for $SP_1$ radiative heat transfer calculations \cite{pinnau-schulze-2007}, and for radiative heat transfer in plasma applications using an optimization problem \cite{fagiano-2016}.
ROMs for thermal analysis of spacecraft have been created with POD-Galerkin and trajectory piecewise-linear methods \cite{qian-wang-song-pant-2015}.
ROMs for grey nonlinear radiation diffusion problems have been formulated with the PGD \cite{alberti-palmer-2018}, POD \cite{alberti-jqsrt-2022} and modal identification method \cite{girault-2021}.
A grey ROM for multigroup TRT problems was developed using POD-based representations of frequency averaged opacities and other coefficients \cite{jc-dya-tans2019}.
POD-based ROMs for radiating fluids have been constructed by applying a linearized estimation of radiative transfer effects based on POD data of non-radiating fluids \cite{soucasse-podvin-riviere-soufiani-2020,soucasse-podvin-riviere-soufiani-2021}.

%
%
In this paper we formulate a projection-based
ROM for the fundamental TRT problem defined by the BTE coupled with the material energy balance (MEB) equation that describes energy exchange between radiation and matter.
This problem models a supersonic radiative flow \cite{Moore-2015}, and serves as a useful platform for the development and testing of new computational methodologies.
The presented  ROM is developed by combination of
low-order equations for the moments of the specific intensity of radiation
and POD-Petrov-Galerkin projection techniques.
The resulting TRT ROM is structure and asymptotic preserving.
The formulation of ROM is based on a multilevel system of equations derived by means of the nonlinear projection approach
which can be interpreted as a nonlinear method of moments
\cite{gol'din-1964,gol'din-1972,Goldin-sbornik-82,dya-vyag-1993,krommes-jpp-2018}.
The BTE is projected sequentially over  elements of the phase space of the problem.
The low-order equations for the angular moments of the intensity are exactly closed with the BTE via the Eddington tensor  (aka quasidiffusion tensor) which is a linear-fractional function of the BTE solution \cite{gol'din-1964,auer-mihalas-1970}.
The moment equations of lowest dimensionality are the effective grey low-order equations  formulated for the angular moments integrated over photon frequency.
The exact  closures for these low-order equations  are defined by various spectrum average coefficients.
The effective grey low-order equations are coupled with multiphysics macroscopic equations at their dimensional scale.
This process constructs a multilevel set of moment equations of the BTE with exact closures and constitutes an initial reduction of dimensionality to the TRT problem with no approximation.
The nonlinear projection approach has been successfully used for dimensionality reduction of  particle transport problems and
developing efficient  iterative methods for solving BTE as well as  multiphysics problems in high-energy density physics and reactor-physics applications \cite{Goldin-sbornik-82,winkler-norman-mihalas-85,PASE-1986,dya-aristova-vya-mm1996,aristova-vya-avk-m&c1999,dya-vyag-nse-2011,at-dya-nse-2014,lrc-dya-pne-2017,dya-jcp-2019,anistratov-jcp-2021}.

The multilevel nonlinear moment equations
can be used as a framework for the development of efficient ROMs for nonlinear multiphysics problems involving
the BTE by formulating data-driven closures to moment equations.
Recently, a ROM was developed for the TRT problem that uses  an equation-free  approximation of the Eddington tensor using direct POD and DMD expansions of Eddington tensor data generated during the offline stage \cite{jc-dya-jqsrt-2023}.
This methodology  was applied to develop a parametrized ROM for multigroup TRT problems in 2D Cartasian geometry.
Another kind of ROM was developed by means of
a POD-Galerkin projection of the BTE. The solution of the projected BTE yields a low-rank
expansion of the
specific radiation intensities in POD modes,
which is then used as an averaging function
to define approximation of the  Eddington tensor \cite{jc-dya-m&c2021}.

In this study, we  apply a POD-Petrov-Galerkin projection to the modified BTE
formulated for a specific intensity normalized by the zeroth intensity moment in angle.
Hereafter we refer to this equation as the normalized BTE (NBTE).
The Eddington tensor linearly depends on the normalized intensity.
This enables us to form a direct expansion of the Eddington tensor in POD modes of the normalized intensities.
By construction, this expansion of the Eddington tensor is defined in terms of angular moments of POD basis.

The modified BTE
formulated with various normalizations has been applied
to derive approximations of the Eddington tensor to be used in VEF methods \cite{pomraning-1969,gnedin-abel-2001} and flux-limited diffusion theories \cite{levermore-pomraning-1981}.
Hybrid Monte-Carlo methods based on the NBTE
 have been able to achieve large variance reductions in the solution for deep penetration problems compared to traditional Monte-Carlo methods \cite{becker-wollaber-larsen-2007}. Anisotropic diffusion ROMs for neutron transport and TRT problems have also been formulated
with a modified BTE whose solution takes the form of a shape function similar to the NBTE solution \cite{johnson-larsen-2011,trahan-larsen-2011}.

This paper presents a detailed analysis of the fundamental properties of the new ROM for nonlinear radiative transfer problems formulated with the POD-Petrov-Galerkin projected NBTE and hierarchy of moment equations.
We look to determine how accurately the ROM can reproduce essential characteristics and physical effects of the target class of
nonlinear TRT problems.
This will depend on the rank of POD expansion of the normalized intensities (and therefore the Eddington tensor), the expansion coefficients for which are obtained from the solution to the projected NBTE.
To  construct a projection for
the NBTE, a spatial discretization scheme  is developed
that is algebraically consistent with  the simple corner balance scheme for the BTE \cite{adams-1997}.
The POD-Petrov-Galerkin projection of the discretized NBTE is derived using a weighted inner product associated with this scheme for the NBTE.
Application of the specific intensity normalized by its zeroth moment yields a POD basis which can be used to construct a direct expansion of the Eddington tensor via superposition of the POD modes.
The normalization of the intensities also creates a naturally bounded function, which helps to mitigate  numerical issues associated with the linear algebra computations needed to find POD basis functions.

The remainder of this paper is organized as follows.
The TRT problem is specified in Sec. \ref{sec:trt}, followed by a detailed overview of the NBTE and
multilevel quasidiffusion method in Secs. \ref{sec:nbte} and \ref{sec:quasidiffusion}.
The discretization scheme for the NBTE is derived in Sec. \ref{sec:nbte_disc}.
The POD-Petrov-Galerkin projection scheme for the NBTE is then derived in Sec. \ref{sec:nbte_proj}.
Discretization of moment equations is described in Sec. \ref{sec:loqd_disc}.
Sec. \ref{sec:rom} formulates the ROM in its entirety.
Numerical results for the ROM are presented in Sec. \ref{sec:numerical_results}, and conclusions are drawn in Sec. \ref{sec:conclusion}.

%
%
\section{Thermal Radiative Transfer Problem} \label{sec:trt}
We consider the TRT problem given by the multigroup BTE \cite{zel-1966,mihalas-FRH-1984}
\begin{gather}    \label{bte_mg}
	\frac{1}{c}\dv{I_g}{t} + \boldsymbol{\Omega}\cdot\grad I_g + \varkappa_g(T)I_g = \varkappa_g(T)B_g(T),\\
	\boldsymbol{r}\in \Gamma,\quad \boldsymbol{\Omega}\in \mathcal{S}_{\Omega},\quad
	g \in \mathbb{N}(G), \quad t \in [0,t^\text{end}]
	\nonumber
\end{gather}
with the boundary condition
\begin{equation}  \label{bte_bc}
	I_g\big|_{\boldsymbol{r} \in\partial\Gamma} = I_g^{\text{in}}, \quad  \boldsymbol{\Omega}\cdot\boldsymbol{n}_\Gamma<0 \, ,
\end{equation}
and the initial condition
\begin{equation}  \label{bte_ic}
	I_g \big|_{t=0}=I_g^0 \, ,
\end{equation}

\noindent and the MEB equation,
which models energy exchange between radiation and matter
\begin{gather}
	\dv{\varepsilon(T)}{t} = \sum_{g=1}^{G} \int_{4\pi} \varkappa_g(T)\big( I_g - B_g(T) \big)\ d\Omega,\quad 	\label{meb} \\
	T|_{t=0}=T^0.  \label{meb_ic}
\end{gather}

\noindent Here $g$ is the index of the frequency group $\nu\in[\nu_g,\nu_{g+1}]$,
 $G$ is the number of frequency groups, $I_g=I_g(\boldsymbol{r},\boldsymbol{\Omega},t)=
\int_{\nu_g}^{\nu_{g+1}}I(\boldsymbol{r},\boldsymbol{\Omega},\nu,t) d \nu$ is the group intensity of radiation,
$\boldsymbol{r}\in\real^3$ is spatial position,
$\boldsymbol{\Omega}$ is the unit vector in the direction of particle motion,
$t$ is time,
$\varepsilon$ is the material energy density, $\varkappa_g$ is the material opacity  for the group $g$,
$T=T(\boldsymbol{r},t)$ is the material temperature,
and $B_g$ is the group Planckian function given by
\begin{equation} \label{eq:planck}
B_g(T) = \frac{2h}{c^2}  \int_{\nu_{g}}^{\nu_{g+1}}   \frac{\nu^3 d\nu}{e^{\frac{h\nu}{kT}}-1}
\end{equation}
$\nu$ is the photon frequency, $c$ is the speed of light, $h$ is Planck's constant, $k$ is the Boltzmann constant.
We denote the unit sphere
$ \mathcal{S}_{\Omega}=
\{ \boldsymbol{\Omega}\in\real^3\ :\ |\boldsymbol{\Omega}|=1 \}$,
$\Gamma\subset\real^3$ is the spatial domain,
$\partial\Gamma$ is the boundary surface of $\Gamma$, and $\boldsymbol{n}_\Gamma$ is the outward-facing unit normal vector to $\partial\Gamma$.
The TRT problem \eqref{bte_mg} and  \eqref{meb}   neglects photon scattering, material motion and heat conduction.

%
%
\section{Normalized Boltzmann Transport Equation} \label{sec:nbte}

To formulate the   multigroup NBTE
we divide the  BTE
(Eq. \eqref{bte_mg})  by
the zeroth angular moment of the group intensity
\begin{equation}\label{phi}
\phi_g=\int_{4\pi}I_g d\Omega \,
\end{equation}
and derive the transport equation for the normalized intensity defined by \cite{becker-wollaber-larsen-2007}
\begin{equation}\label{n-intensity}
\bar{I}_g=\frac{I_g}{\phi_g} \, .
\end{equation}
The resulting multigroup NBTE  for $\bar{I}_g$  given by
\begin{gather} \label{eq:nbte}
	\frac{1}{c}\dv{\bar{I}_g}{t} + \boldsymbol{\Omega}\cdot\boldsymbol{\nabla} \bar{I}_g + \hat{\varkappa}_{g}(T,\phi_g) \bar{I}_g = \varkappa_{g}(T)\bar{B}_g(T,\phi_g),
  \ g \in \mathbb{N}(G) \, ,
\end{gather}
\noindent  where
\begin{equation}
\hat{\varkappa}_{g} = \varkappa_{g} + \frac{1}{c}\dv{\ln(\phi_g)}{t} + \boldsymbol{\Omega}\cdot\boldsymbol{\nabla}\ln(\phi_g) \, ,
\end{equation}
and
\begin{equation}
\bar{B}_g = \frac{B_g}{\phi_g} \, .
\end{equation}
The boundary and initial conditions for  the NBTE are defined as follows:
\begin{equation}  \label{nbte_bc}
	\bar{I}_g\big|_{\boldsymbol{r} \in\partial\Gamma} = \frac{I_g^{\text{in}}}{\phi_g}\bigg|_{\boldsymbol{r} \in\partial\Gamma} \, , \quad  \boldsymbol{\Omega}\cdot\boldsymbol{n}_\Gamma<0 \, ,
\end{equation}
 \begin{equation}  \label{nbte_ic}
 	\bar{I}_g \big|_{t=0}=\frac{I_g^0}{\phi_g} \bigg|_{t=0} \, .
 \end{equation}

The structure of the NBTE is similar to that of the BTE.
 $\hat \varkappa_g$ can be interpreted as a modified opacity.
  $\bar{B}_g $ is the scaled Planckian function.
  We note that the solution of the NBTE is an angular  shape function in the group $g$ such that
 \begin{equation}  \label{nbte_ic}
 	\int_{4 \pi} \bar{I}_g  d \Omega = 1 \, .
 \end{equation}

%
%
\section{The Multilevel System of Moment Equations \label{sec:quasidiffusion}}
A hierarchy  of moment equations is formulated  by projecting the BTE
(Eq. \eqref{bte_mg})  onto a sequence of subspaces  in two stages \cite{Goldin-sbornik-82,PASE-1986}.
In first stage of projection, the BTE is integrated over
$\boldsymbol{\Omega}\in \mathcal{S}_{\Omega}$
with weights $1$ and $\boldsymbol{\Omega}$
to obtain  the
multigroup  low-order equations  given by \cite{gol'din-1964,gol'din-1972,auer-mihalas-1970}
\begin{subequations} \label{gp-loqd}
	\begin{gather}
		\dv{E_g}{t} + \grad\cdot\boldsymbol{F}_g + c\varkappa_g(T)E_g = 4\pi\varkappa_g(T)B_g(T), \label{gp_ebal}\\
		\frac{1}{c}\dv{\boldsymbol{F}_g}{t} +
		c\grad\cdot (\boldsymbol{\mathfrak{f}}_g E_g)
 		+ \varkappa_g(T) \boldsymbol{F}_g = 0, \label{gp_mbal}
	\end{gather}
\end{subequations}

\noindent whose solution is the multigroup radiation energy density $E_g=\frac{1}{c}\int_{4\pi}I_g d\Omega$ and flux $\boldsymbol{F}_g=\int_{4\pi}\boldsymbol{\Omega}I_g d\Omega$.
The moment equations \eqref{gp-loqd} are exactly closed with the Eddington tensor
\begin{equation}
	\boldsymbol{\mathfrak{f}}_g   \stackrel{\Delta}{=}
\frac{ \int_{4\pi} \boldsymbol{\Omega}\otimes\boldsymbol{\Omega}  {I}_g d\Omega}
{ \int_{4\pi}  {I}_g d\Omega} \, .
\end{equation}
Hereafter we refer to the moment equations \eqref{gp-loqd} as the multigroup low-order quasidiffusion (LOQD) equations
(aka VEF equations).
Note that  the zeroth intensity moment is defined by the solution of the
LOQD equations \eqref{gp-loqd} as follows:
\begin{equation}
 \phi_g=cE_g \, .
\end{equation}

The components of the Eddington tensor are linear fractional functions of ${I}_g$
and hence nonlinearly depend on the solution of the BTE (Eq. \eqref{bte_mg}).
The multigroup LOQD equations can be closed with the NBTE via application of the Eddington tensor defined in terms of the normalized intensity $\bar I_g$:
  \begin{equation}
\boldsymbol{\mathfrak{f}}_g   =
 \mathcal{H} \bar I_g   \, .
\end{equation}
where the operator is defined by
\begin{equation}\label{H-oper}
 \mathcal{H} \cdot \stackrel{\Delta}{=}  \int_{4\pi} d\Omega  \boldsymbol{\Omega}\otimes\boldsymbol{\Omega} \cdot \, ,
\end{equation}
In this case,  $\boldsymbol{\mathfrak{f}}_g$
linearly depends on  the solution of the NBTE  (Eq. \eqref{eq:nbte}).
Formulation of the closure based on a linear relationship with the high-order solution has a distinct advantage.
When the ROM is formulated in Sec. \ref{sec:rom}, this feature is what
enables us to formulate a direct expansion of the Eddington tensor in terms of the POD  basis functions of $\bar I_g$.

We apply the boundary conditions for the multigroup LOQD equations \eqref{gp-loqd} defined as follows \cite{miften-larsen-1993}:
\begin{equation}
	\boldsymbol{n}_\Gamma\cdot\boldsymbol{F}_g\big|_{r\in\partial\Gamma} =
c \beta_gE_g\big|_{\boldsymbol{r}\in\partial\Gamma} + 2F_g^\text{in},\quad
\end{equation}
where the boundary factor defined with the solution of the NBTE is given by
\begin{equation}
\beta_g  = \mathcal{B} \bar{I}_g\big|_{ \boldsymbol{r}\in\partial\Gamma} \, ,
\end{equation}
where
\begin{equation} \label{beta-oper}
\mathcal{B} \cdot   \stackrel{\Delta}{=} \int_{4\pi} d\Omega |\boldsymbol{n}_\Gamma\cdot\boldsymbol{\Omega}|\  \cdot      \, .
\end{equation}
This boundary factor also has a linear
representation in terms of $\bar{I}_g$.
The initial conditions are
\begin{equation}
	E_g\big|_{t=0} = E_g^0,\quad \boldsymbol{F}_g\big|_{t=0} = \boldsymbol{F}_g^0,
\end{equation}
where
\begin{equation}
	E_g^0=\frac{1}{c}\int_{4\pi}I_g^0 d\Omega,\quad \boldsymbol{F}_g^0=\int_{4\pi}\boldsymbol{\Omega}I_g^0 d\Omega, \quad {F}_g^\text{in}=\int_{\boldsymbol{n}_\Gamma\cdot\boldsymbol{\Omega}<0}(\boldsymbol{n}_\Gamma\cdot\boldsymbol{\Omega}) I_g \big|_{r\in\partial\Gamma}\ d\Omega.
\end{equation}

In the second stage of projection, Eqs. \eqref{gp-loqd} are summed over all frequency groups to derive the
{\it effective grey} LOQD equations
 given by \cite{PASE-1986}
\begin{subequations}	
	\begin{gather}
		\dv{E}{t} + \boldsymbol{\nabla}\cdot\boldsymbol{F} +
		c\ip{\varkappa}_E E
		=
		c\ip{\varkappa}_B
		 a_RT^4,\\
		\frac{1}{c}\dv{\boldsymbol{F}}{t} + c\boldsymbol{\nabla} \cdot
		(  \ip{\boldsymbol{\mathfrak{f}}}_{E} E)  +
		\ip{\mathbf{K}}_{F} \boldsymbol{F}
		+ \bar{\boldsymbol{\eta}}E = 0,
		\label{eq:gr-loqd}
	\end{gather}
	\label{gqd}
\end{subequations}

\noindent whose solution is the total radiation energy density $E=\sum_{g=1}^GE_g$ and flux $\boldsymbol{F}=\sum_{g=1}^G\boldsymbol{F}_g$. The Eqs. \eqref{gqd} are exactly closed with
the spectrum-averaged Eddington tensor
\begin{equation}
\ip{\boldsymbol{\mathfrak{f}}}_E = \frac{1}{\sum_{g=1}^G E_g}
\sum_{g=1}^G \boldsymbol{\mathfrak{f}}_g  E_g \, , \quad
\end{equation}
and the averaged opacities
\begin{equation}
\ip{\varkappa}_E = \frac{\sum_{g=1}^G \varkappa_g E_g}{\sum_{g=1}^G E_g} \, , \quad
\ip{\varkappa}_B = \frac{\sum_{g=1}^G \varkappa_g B_g}{\sum_{g=1}^G B_g} \, , \quad
\ip{\varkappa}_{F_{\gamma}} = \frac{\sum_{g=1}^G \varkappa_g | F_{\gamma,g}|}{\sum_{g=1}^G  | F_{\gamma,g}|} \, ,
\end{equation}
\begin{equation}
\ip{\mathbf{K}}_{F}
=
\bigoplus_{\gamma=1}^{\rho}\ip{\varkappa}_{F_{\gamma}}  \, ,
\end{equation}
where $\rho$ is the number of spatial dimensions.
The compensation term
is defined by
\begin{equation}
	\bar{\boldsymbol{\eta}}= \frac{\sum_{g=1}^{G} (\varkappa_g -  \ip{\mathbf{K}}_{F}    )\boldsymbol{F}_g}{\sum_{g=1}^GE_g}  \, .
\end{equation}

\noindent Boundary conditions for Eqs. \eqref{gqd} are defined by
\begin{equation}
	\boldsymbol{n}_\Gamma\cdot\boldsymbol{F}\big|_{r\in\partial\Gamma} =
 c \ip{\beta}_E E\big|_{r\in\partial\Gamma} + 2F^\text{in},
\end{equation}
where
\begin{equation}
\ip{\beta}_E   = \frac{\sum_{g=1}^G \beta_g E_g}{\sum_{g=1}^G E_g} \, ,
\end{equation}
and the initial conditions are
\begin{equation}
	E\big|_{t=0} = E^0,\quad \boldsymbol{F}\big|_{t=0} = \boldsymbol{F}^0,
\end{equation}
where
\begin{equation}
	E^0=\sum_{g=1}^G E_g^0,\quad \boldsymbol{F}^0=\sum_{g=1}^G \boldsymbol{F}_g^0, \quad {F}^\text{in}=\sum_{g=1}^G{F}_g^\text{in}.
\end{equation}

The MEB equation
(Eq. \eqref{meb}) is coupled with the effective grey LOQD equations after being recast in effective grey form
\begin{gather}
	\dv{\varepsilon(T)}{t} = c \Big(\ip{\varkappa}_E E -  \ip{\varkappa}_Ba_RT^4\Big) ,
	\label{grey_meb}
\end{gather}

\noindent creating the coupled {\it effective grey problem}  (Eqs. \eqref{gqd} and \eqref{grey_meb}).

To summarize, the multilevel system of  equations for solving the TRT problem
 consists  of the following high-order and low-order equations in operator form:
\begin{itemize}
\item the high-order equation for $\bar{I}_g$ defined by the   multigroup NBTE (Eq. \eqref{eq:nbte})
  \begin{equation}\label{nbte-o-form}
  \frac{1}{c}\frac{\partial \bar{I}_g}{\partial t} + \mathcal{\bar L}_g   \bar{I}_g  =  \bar q_g   \, ,
  \ g \in \mathbb{N}(G) \, ,
 \end{equation}
  \begin{equation}
\mathcal{\bar L}_g = \mathcal{\bar L}_g [T,\phi_g]  \, , \quad
  \bar q_g =\bar q_g (T,\phi_g)  \, ,
  \end{equation}
  \item the  multigroup LOQD equations for $E_g$ and $\boldsymbol{F}_g$  (Eqs. \eqref{gp-loqd})
\begin{equation} \label{mloqd-o-form}
 \frac{\partial \boldsymbol{\varphi}_g}{\partial t} + \mathcal{M}_g \boldsymbol{\varphi}_g  = \mathbf{q}_g \, , \
\boldsymbol{\varphi}_g  =  \begin{pmatrix} E_g \\ \boldsymbol{F}_g\end{pmatrix}\, ,
 \ g \in \mathbb{N}(G) \, ,
\end{equation}
\begin{equation}\label{mloqd-oper}
\mathcal{M}_g = \mathcal{M}_g[T, \boldsymbol{\mathfrak{f}}_g, \beta_g] \, , \quad
  \mathbf{q}_g=\mathbf{q}_g[T] \, ,
 \end{equation}
\item  the  effective grey  LOQD equations for $E$ and $\boldsymbol{F}$ (Eqs. \eqref{gqd})   given by
\begin{equation} \label{gloqd-o-form}
 \frac{\partial \boldsymbol{\hat \varphi}}{\partial t} +  \mathcal{\hat M} \boldsymbol{\hat \varphi} = \mathbf{\hat q} \, ,
\quad
 \boldsymbol{\hat \varphi} = \begin{pmatrix} E_g \\ \boldsymbol{F}_g\end{pmatrix} \, ,
\end{equation}
\begin{equation}\label{gloqd-oper}
 \mathcal{\hat  M}  = \mathcal{\hat M}[ T,\boldsymbol{\varphi}_g, \boldsymbol{\mathfrak{f}}_g, \beta_g ] \, , \quad
 \mathbf{\hat q}=\mathbf{\hat q}(T) \, ,
\end{equation}
  \item the   effective grey MEB  equation  (Eq. \eqref{grey_meb}) for $T$ and $E$ defined  by
\begin{equation}\label{grey-meb-o-form}
\frac{\partial \varepsilon(T)}{\partial t} =    \mathcal{P}  \boldsymbol{\hat \varphi} +   \tilde q  \, ,
\end{equation}
\begin{equation}\label{grey-meb-oper}
\mathcal{P}=\mathcal{P}[T,  \boldsymbol{\varphi}_g] \, , \quad \tilde q = \tilde q(T) \, .
\end{equation}
\end{itemize}

%
%
\section{Discretization of the Normalized Boltzmann Transport Equation}\label{sec:nbte_disc}

\begin{figure}
\begin{centering}
\begin{picture}(300,200)
\drawline(25,25)(150,25) 
\drawline (275,175)(25,175)(25,25)
\drawline(275,100)(275,175) 
\drawline(150,25)(150,175)
\drawline(25,100)(275,100)
 \dashline{3}(87.5,25)(87.5,175)
\dashline{3}(25,137.5)(275,137.5)
\dashline{3}(25,62.5)(150,62.5)
\dashline{3}(212.5,100)(212.5,175)
\put(87.5,137.5){\circle{5}}
\put(87.5,137.5){\circle{6}}
\put(90,140){$i$}
\put(212.5,137.5){\circle{5}}
\put(212.5,137.5){\circle{6}}
\put(215,140){$i^{\prime}$}
\put(87.5,62.5){\circle{5}}
\put(87.5,62.5){\circle{6}}
\put(90,65){$i^{\prime\prime}$}
\put(118.75,118.75){\circle*{3}}
\put(110.0,116.75){{\footnotesize $\alpha$}}
\put(150,118.75){\circle*{3}}
\put(134,116.75){{\footnotesize $\alpha$+}}
\put(87.5,118.75){\circle*{3}}
\put(90.5,116.75){{\footnotesize $\alpha$-$\frac{1}{2}$}}
\put(118.75,100){\circle*{3}}
\put(114.75,103.5){{\footnotesize $\alpha$--}}
\put(56.25,118.75){\circle*{3}}
\put(46.25,126.25){ {\footnotesize $\alpha$--1}}
\put(118.75,156.25){\circle*{3}}
\put(110.25,162.25){{\footnotesize $\alpha$+1}}
\put(118.75,137.5){\circle*{3}}
\put(110.25,127.5){{\footnotesize $\alpha$+$\frac{1}{2}$}}
\put(118.75,81.25){\circle*{3}}
\put(113.75,67.25){{\footnotesize $\alpha^{\prime\prime}$}}
\put(181.25,118.75){\circle*{3}}
\put(177.25,126.25){{\footnotesize $\alpha^{\prime}$}}
\end{picture}
\caption{Notations for the SCB scheme in cells and corners.}
\label{scb_corner}
\end{centering}
\end{figure}
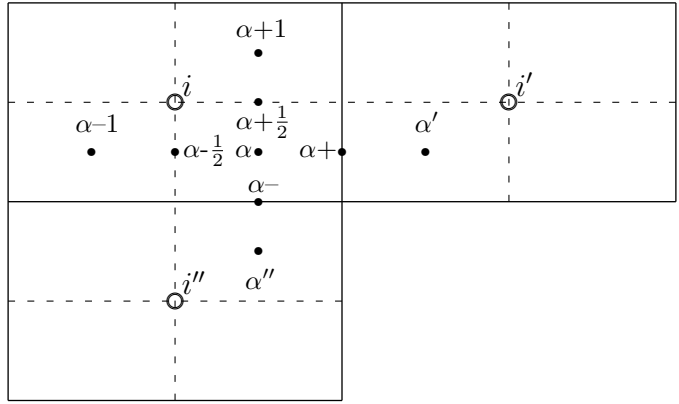
To derive the discretization of the NBTE (Eq. \eqref{eq:nbte}),
we start from approximating the BTE
(Eq. \eqref{bte_mg})
The method of discrete-ordinates (MDO) is used to discretize the BTE with respect to the angular variable.
The backward-Euler (BE) time integration method is the applied discretization in time.
This yields the BTE in semi-continuous form  given by
\begin{gather} \label{bte-be}
	\frac{ 1}{c \Delta t^n}
\Big(  I_{g,m}^n -  I_{g,m}^{n-1} \Big) +
 \boldsymbol{\Omega}_m \cdot \boldsymbol{\nabla} I _{g,m}^n +  \varkappa_{g}^n I _{g,m}^n =  \varkappa_{g}^n  B_g^n \, ,\\
g \in \mathbb{N}(G)\, ,\ n \in \mathbb{N}(N) \, ,  \ m \in \mathbb{N}(M) \, ,\nonumber
\end{gather}
 where $m$ is the index of  discrete direction, $n$ is the index of the time step,  $\Delta t^n = t^n - t^{n-1}$,
 $I_{g,m}^n = I_{g} (\boldsymbol{r},\boldsymbol{\Omega}_m, t^n)$.
To approximate Eq. \eqref{bte-be} in space  we apply the simple corner balance (SCB) scheme \cite{adams-1997}.
We consider 2D problems  with  $T$  defined by its cell-average values.
Fig. \ref{scb_corner} illustrates  elements of rectangular spatial grids showing
the corner $\alpha$ in the $i^{th}$ spatial cell $\mathcal{D}_i$,
its neighbouring corners  $\alpha-1$ and $\alpha+1$ inside the cell, and
corners $\alpha^{\prime}$  and $\alpha^{\prime\prime}$
in the adjacent cells $\mathcal{D}_{i^{\prime}} $ and $\mathcal{D}_{i^{\prime\prime}}$
that share a vertex with corner $\alpha$.
In an  arbitrary polygon cell  $\mathcal{D}_i$ with $C$ number   of corners,
the  SCB scheme for the  BTE
(Eq. \eqref{bte_mg})
is defined by a set of balance equations in each corner given by
 \begin{multline} \label{scb}
  \frac{S_{i,\alpha}}{c\Delta t^n} \Big(I_{g,m,i,\alpha}^n -   I_{g,m,i,\alpha}^{n-1} \Big)  \\
   + \mathbf{\Omega}_m \cdot \Big( \boldsymbol{\ell}_{i,\alpha+} I_{g,m,i,\alpha+}^n  + \boldsymbol{\ell}_{i,\alpha-} I_{g,m,i,\alpha-}^n
 + \boldsymbol{\ell}_{i,\alpha+1/2} I_{g,m,i,\alpha+1/2}^n
  + \boldsymbol{\ell}_{i,\alpha-1/2} I_{g,m,i,\alpha-1/2}^n \Bigr)  \\
  +\varkappa_{g,i}^n  I_{g,m,i,\alpha}^n S_{i,\alpha}
 =   \varkappa_{g,i}^n B_{g,i}^n S_{i,\alpha} \, ,
 \end{multline}
  \begin{equation}
 \boldsymbol{\ell}_{i,\alpha \pm} = \ell_{i,\alpha \pm}  \boldsymbol{n}_{i,\alpha \pm} \, , \quad
  \boldsymbol{\ell}_{i,\alpha \pm1/2} = \ell_{i,\alpha \pm1/2}  \boldsymbol{n}_{i,\alpha \pm1/2} \, ,
 \end{equation}
 \begin{equation*}
 g \in \mathbb{N}(G) \, ,  \ m \in \mathbb{N}(M) \, ,  \ i \in \mathbb{N}(X) \, ,
 \alpha \in \mathbb{N}(C) \, ,
  \end{equation*}
where
$\ell_{i,\alpha \pm}$ and $\ell_{i,\alpha \pm1/2}$ are lengths of corner edges
$\alpha \pm$ and $\alpha \pm1/2$, respectively,
$\boldsymbol{n}_{i,\alpha\pm}$ and $\boldsymbol{n}_{i,\alpha \pm1/2}$ are outward unit normals at these edges,
$S_{i,\alpha}$ is the area of corner $\alpha$,
$I_{g,m,i,\alpha}^n$  is the corner radiation intensity,
 $I_{g,m,i,\alpha\pm}^n$ and $I_{m,i,\alpha\pm1/2}^n$ are corner-edge  values.
The  SCB auxiliary relations are defined by
simple  averaging for corner surfaces inside the cell
\begin{equation} \label{scb-a1}
I_{g,m,i,\alpha\pm1/2}^n = \frac{1}{2}(I_{g,m,i,\alpha}^n + I_{g,m,i,\alpha \pm 1}^n) \, ,
\end{equation}
and by upstream values at interfaces between cells
\begin{equation} \label{scb-a2}
I_{g,m,i,\alpha\pm}^n=
\begin{cases}
I_{g,m,i,\alpha}^n\, , \ & \mathbf{\Omega}_m  \cdot \boldsymbol{n}_{i,\alpha\pm} > 0 \, ,\\
 I_{g,m,i^*,\alpha^*}^n\, , \ & \mathbf{\Omega}_m   \cdot \boldsymbol{n}_{i,\alpha\pm} < 0 \, ,
\end{cases}
\ \
i^*=
\begin{cases}
i^{\prime} & \mbox{for} \ \alpha+ \, , \\
i^{\prime\prime} & \mbox{for} \ \alpha- \, ,
\end{cases}
\ \
\alpha^*=
\begin{cases}
\alpha^{\prime} & \mbox{for} \ \alpha+ \, , \\
\alpha^{\prime\prime} & \mbox{for} \ \alpha\!- \, .
\end{cases}
\end{equation}
Using the auxiliary conditions \eqref{scb-a1} and  \eqref{scb-a2} in the corner balance equations \eqref{scb}
the SCB scheme can be cast in terms of corner intensities  in the following operator form:
\begin{equation} \label{scb-o-form}
  \mathcal{T}^n \Big(\boldsymbol{I}_h^n -   \boldsymbol{I}_h^{n-1} \Big)
  +\mathcal{L}_h^n\boldsymbol{I}_h^n =\mathbf{Q}_h^n \, ,
\end{equation}
 where $\boldsymbol{I}_h^n \in\   \real^{d}$ is the vector of cell-corner  intensities over all phase space at $t^n$,
 $d=GMXC$ is the number of degrees of freedom of the NBTE solution at a time layer,
  $\mathbf{Q}_h^n \in  \real^{d} $ is the vector of corner  Planckian source terms, and the operator $\mathcal{T}^n$ multiplies its operand elementwise in cell, corner locations $(i,\alpha)$ by $\frac{S_{i,\alpha}}{c\Delta t^n}$.

On the next step of the derivation, we   define the  grids functions of the normalized intensity in the $i^{th}$ cell as follows:
\begin{multline} \label{eq:nbte_gfunc}
	\bar I_{g,m,i,\alpha}^n = \frac{I_{g,m,i,\alpha}^n}{\phi_{g,i,c}^n} \, , \quad
	\bar I_{g,m,i,\alpha \pm 1/2}^n = \frac{I_{g,m,i,\alpha \pm 1/2}^n}{\phi_{g,i,c}^n} \, , \quad \\
	\bar I_{g,m,i,\alpha+}^n = \frac{I_{g,m,i,\alpha+}^n}{\phi_{g,i,\alpha}^n} \, , \quad
	\bar I_{g,m,i,\alpha-}^n = \frac{I_{g,m,i,\alpha-}^n}{\phi_{g,i,\alpha-1}^n} \, ,
\end{multline}
and the normalized cell-average Planckian function
\begin{equation}
	\bar B_{g,i}^n = \frac{B_{g,i}^n}{\phi_{g,i,c}^n} \, ,
\end{equation}
where
\begin{equation}
	\phi_{g,i,c}^n = \frac{1}{\sum_{\alpha=1}^{C}S_{i,\alpha}}\sum_{\alpha=1}^{C}  S_{i,\alpha}\sum_{m=1}^{M}I_{g,m,i,\alpha}^nw_m
\end{equation}
is the cell-average zeroth moment  of the intensity in $\mathcal{D}_i$ and
\begin{equation}
\phi_{g,i,\alpha}^n = \frac{1}{2 }\sum_{m=1}^{M}  \big( I_{g,m,i,\alpha+}^n + I_{g,m,i,(\alpha+1)-}^n \big)w_m \,
\end{equation}
is the face-average zeroth moment  of the intensity over the face shared by corners $\alpha$ and $\alpha+1$.
Here $w_m$ is the quadrature weight.
The equation \eqref{eq:nbte_gfunc} leads to
\begin{multline} \label{eq:nbte_gfunc-2}
	I_{g,m,i,\alpha}^n =\bar I_{g,m,i,\alpha}^n \phi_{i,c}^n \, , \quad
	 I_{g,m,i,\alpha \pm 1/2}^n =  \bar I_{g,m,i,\alpha \pm 1/2}^n  \phi_{i,c}^n  \, , \quad \\
	I_{g,m,i,\alpha+}^n = \bar I_{g,m,i,\alpha+}^n \phi_{i,\alpha}^n  \, , \quad
	I_{g,m,i,\alpha-}^n = \bar I_{g,m,i,\alpha-}^n \phi_{i,\alpha-1}^n \, .
\end{multline}
Substituting Eq. \eqref{eq:nbte_gfunc-2} in the corner balance equation \eqref{scb}
 and dividing the resulting equation by
$\phi_{i,c}$ we obtain the
discretized NBTE in the $i^{th}$ cell given by
 \begin{multline}
 \frac{S_{i,\alpha}}{c \Delta t^n} \Big(\bar I_{g,m,i,\alpha}^n -  \frac{\phi_{g,i,c}^{n-1}}{\phi_{g,i,c}^n} \bar I_{g,m,i,\alpha}^{n-1} \Big) +\\
	\mathbf{\Omega}_m \cdot \Big( \boldsymbol{\ell}_{i,\alpha+} \frac{\phi_{g,i,\alpha}^n}{\phi_{g,i,c}^n} \bar I_{g,m,i,\alpha+}^n
	+ \boldsymbol{\ell}_{i,\alpha-}  \frac{\phi_{g,i,\alpha-1}^n}{\phi_{g,i,c}^n} \bar I_{g,m,i,\alpha-}^n
	+ \boldsymbol{\ell}_{i,\alpha+1/2} \bar I_{g,m,i,\alpha+1/2}^n
	+ \boldsymbol{\ell}_{i,\alpha-1/2} \bar I_{g,m,i,\alpha-1/2}^n \Big)   \\
	+    \varkappa_{g,i}^n S_{i,\alpha} \bar I_{g,m,i,\alpha}^n
	=   \varkappa_{g,i}^n \bar B_{g,i}^n S_{i,\alpha} \, ,
	\label{scb-norm}
\end{multline}
 \begin{equation*}
 g \in \mathbb{N}(G) \, ,  \ m \in \mathbb{N}(M) \, ,  \ i \in \mathbb{N}(X) \, , \
 \alpha \in \mathbb{N}(C) \, , \ n \in \mathbb{N}(N) \, .
  \end{equation*}
Introducing Eq. \eqref{eq:nbte_gfunc-2} in Eqs. \eqref{scb-a1} and \eqref{scb-a2},
 we get the auxiliary relations defined by
\begin{subequations}\label{scb-norm-aux}
	\begin{equation}
		\bar I_{g,m,i,\alpha\pm1/2}^n = \frac{1}{2}(\bar I_{g,m,i,\alpha}^n  + \bar I_{g,m,i,\alpha \pm 1}^n ) \, ,
		\label{scb-norm-a1}
	\end{equation}
	\begin{multline}
		\bar I_{g,m,i,\alpha+}^n =
		\begin{cases}
			\frac{\phi_{g,i,c}^n }{\phi_{g,i,\alpha }^n } \bar I_{g,m,i,\alpha}^n \, , \ & \mathbf{\Omega}_m  \cdot \boldsymbol{n}_{i,\alpha+} > 0 \, ,
\medskip \\
			\frac{\phi_{g,i^{\prime},c}^n }{\phi_{g,i, \alpha }^n} \bar I_{g,m,i^{\prime},\alpha^{\prime}}^n \, , \ & \mathbf{\Omega}_m   \cdot \boldsymbol{n}_{i,\alpha+} < 0 \, ,
		\end{cases}
\quad \\
		\bar I_{g,m,i,\alpha-}^n =
		\begin{cases}
			\frac{\phi_{g,i,c}^n }{\phi_{g,i, \alpha-1 }^n } \bar I_{g,m,i,\alpha}^n \, , \ & \mathbf{\Omega}_m  \cdot \boldsymbol{n}_{i,\alpha-} > 0 \, ,
\medskip \\
			\frac{\phi_{g,i^{\prime\prime},c}^n }{\phi_{g,i, \alpha-1 }} \bar I_{g,m,i^{\prime\prime},\alpha^{\prime\prime}}^n \, , \ & \mathbf{\Omega}_m   \cdot \boldsymbol{n}_{i,\alpha-} < 0 \, .
		\end{cases}
		\label{scb-norm-a2}
	\end{multline}
\end{subequations}
Plugging the auxiliary equations  \eqref{scb-norm-aux} in Eq.  \eqref{scb-norm} we obtain the
 corner-balance equations in terms of the corner normalized intensities
 \begin{subequations}\label{scb-norm-c}
 \begin{multline}
 \frac{S_{i,\alpha}}{c \Delta t^n} \Big(\bar I_{g,m,i,\alpha}^n -  \frac{\phi_{g,i,c}^{n-1}}{\phi_{g,i,c}^n} \bar I_{g,m,i,\alpha}^{n-1} \Big) +\\
	\mathbf{\Omega}_m \cdot \Big(
\boldsymbol{\ell}_{i,\alpha+}  \bar I_{g,m,i,\alpha}^n
	+ \boldsymbol{\ell}_{i,\alpha-}  \frac{\phi_{g,i^{\prime\prime},c}^n}{\phi_{g,i,c}^n} \bar I_{g,m,i^{\prime\prime},\alpha^{\prime\prime}}^n
	+ \frac{1}{2}\boldsymbol{\ell}_{i,\alpha+1/2}\big(\bar I_{g,m,i,\alpha}^n + \bar I_{g,m,i,\alpha+1}^n\big)
	+ \frac{1}{2}\boldsymbol{\ell}_{i,\alpha-1/2} \big(\bar I_{g,m,i,\alpha}^n + \bar I_{g,m,i,\alpha-1}^n\big) \Big)   \\
	+    \varkappa_{g,i}^n S_{i,\alpha} \bar I_{g,m,i,\alpha}^n
	=   \varkappa_{g,i}^n \bar B_{g,i}^n S_{i,\alpha} \, , \quad \\
\mathbf{\Omega}_m  \cdot \boldsymbol{n}_{i,\alpha+} > 0 \, , \quad
\mathbf{\Omega}_m   \cdot \boldsymbol{n}_{i,\alpha-} < 0 \, ,
	\label{scb-norm-c1}
\end{multline}
 \begin{multline}
 \frac{S_{i,\alpha}}{c \Delta t^n} \Big(\bar I_{g,m,i,\alpha}^n -  \frac{\phi_{g,i,c}^{n-1}}{\phi_{g,i,c}^n} \bar I_{g,m,i,\alpha}^{n-1} \Big) +\\
	\mathbf{\Omega}_m \cdot \Big(
      \boldsymbol{\ell}_{i,\alpha+} \frac{\phi_{g,i^{\prime},c}^n}{\phi_{g,i,c}^n} \bar I_{g,m,i^{\prime},\alpha^{\prime}}^n
	+ \boldsymbol{\ell}_{i,\alpha-}   \bar I_{g,m,i,\alpha}^n
	+ \frac{1}{2}\boldsymbol{\ell}_{i,\alpha+1/2} \big(\bar I_{g,m,i,\alpha}^n + \bar I_{g,m,i,\alpha+1}^n\big)
	+ \frac{1}{2}\boldsymbol{\ell}_{i,\alpha-1/2} \big(\bar I_{g,m,i,\alpha}^n + \bar I_{g,m,i,\alpha-1}^n\big)\Big)   \\
	+    \varkappa_{g,i}^n S_{i,\alpha} \bar I_{g,m,i,\alpha}^n
	=   \varkappa_{g,i}^n \bar B_{g,i}^n S_{i,\alpha} \,  \quad \\
\mathbf{\Omega}_m  \cdot \boldsymbol{n}_{i,\alpha+} < 0 \, ,\quad
\mathbf{\Omega}_m   \cdot \boldsymbol{n}_{i,\alpha-} > 0 \, .
	\label{scb-norm-c2}
\end{multline}
\end{subequations}

The formulated spatial  discretization given in Eqs. \eqref{scb-norm} and \eqref{scb-norm-aux}  of the semi-continuous NBTE
(Eq. \eqref{bte-be}) is algebraically consistent with the SCB scheme for  the BTE.
We note the scheme \eqref{scb-norm} can be interpreted as an approximation of
 the NBTE in the alternative form   given by
\begin{equation} \label{nbte-alt}
	\frac{1}{c}\dv{(\phi_g\bar{I}_g)}{t} +
 \boldsymbol{\Omega}\cdot \boldsymbol{\nabla}(\phi_g\bar{I}_g)+  \varkappa_{g}\phi_g\bar{I}_g =  \varkappa_{g} B_g \, .
\end{equation}

The solution of the discretized NBTE
(Eqs. \eqref{scb-norm} and \eqref{scb-norm-aux}) is used as the basis for a POD-Petrov-Galerkin projection scheme.
During this process, the elements of these discretized equations are summed over all grid points in the involved discrete phase space. For large problems these huge summations (over space, frequency, angle) can pose a numerical challenge when the individual elements span several orders of magnitude, as is the case for our target class of complex multiphysics problems.
Fig. \ref{fig:fc_opac_vs-nu} (see Sec. \ref{sec:results}) displays the range of values that are typical for opacities in the considered TRT problems.
Opacities for low frequency radiation have the largest magnitudes, and opacities for high frequency radiation have the smallest magnitudes. In the types of problems considered here
the physics is primarily driven by radiation in the high frequency range. This means the majority of cancellation errors that occur in full-phase space summations will affect the most important piece of the solution.
To minimize the prominent numerical issues with POD-Petrov-Galerkin projection for these TRT problems the  corner-balance NBTE equations \eqref {scb-norm} are scaled by the opacity $\varkappa_{g,i}^n$ in each cell to obtain:

 \begin{multline}
 \frac{S_{i,\alpha}}{c \Delta t^n\varkappa_{g,i}^n} \Big(\bar I_{g,m,i,\alpha}^n -\frac{\phi_{i,c}^{n-1}}{\phi_{i,c}^n}  \bar I_{g,m,i,\alpha}^{n-1} \Big) +\\
	\frac{1}{\varkappa_{g,i}^n}\mathbf{\Omega}_m \cdot \Big( \boldsymbol{\ell}_{i,\alpha+} \frac{\phi_{g,i,\alpha}^n}{\phi_{g,i,c}^n} \bar I_{g,m,i,\alpha+}^n
	+ \boldsymbol{\ell}_{i,\alpha-}  \frac{\phi_{g,i,\alpha-1}^n}{\phi_{g,i,c}^n} \bar I_{g,m,i,\alpha-}^n
	+ \boldsymbol{\ell}_{i,\alpha+1/2} \bar I_{g,m,i,\alpha+1/2}^n
	+ \boldsymbol{\ell}_{i,\alpha-1/2} \bar I_{g,m,i,\alpha-1/2}^n \Bigr)   \\
	+    S_{i,\alpha} \bar I_{g,m,i,\alpha}^n
	=   \bar B_{g,i}^n S_{i,\alpha} \, .
	\label{scb-norm-scaled}
\end{multline}
This scaling acts to reduce the magnitude of coefficients and right hand side of the discrete NBTE in low-frequency groups, while increasing them for the high-frequency groups.

%
%
\section{POD-Petrov-Galerkin Projection of the Normalized Boltzmann Transport Equation}\label{sec:nbte_proj}

The discretized NBTE  (Eqs. \eqref{scb-norm} and \eqref{scb-norm-aux})  reduced to
the set of equations for the  corner normalized intensities    (Eq. \eqref{scb-norm-c})
 has the following operator form:
\begin{equation} \label{eq:disc-nbte}
\big( \mathcal{T}^n_h\boldsymbol{\bar{I}}_h^n - \mathcal{\bar T}_h^n\boldsymbol{\bar{I}}_h^{n-1} \big) + \mathcal{\bar L}_h\boldsymbol{\bar{I}}_h^n = \mathbf{\bar Q}_h^n \, ,
 \end{equation}
 \begin{equation}
  \mathcal{\bar T}_h^n =  \mathcal{\bar T}_h^n[\boldsymbol{\phi}_h^n,\boldsymbol{\phi}_h^{n-1}] \, , \quad
  \mathcal{\bar L}_h = \mathcal{\bar L}_h[\boldsymbol{\phi}_h^n] \, , \quad \mathbf{\bar Q}_h^n = \mathbf{\bar Q}_h^n(\boldsymbol{\phi}^n) \, .
 \end{equation}
Here $\boldsymbol{\bar{I}}_h^n\in\   \real^{d}$ is the vector of corner normalized intensities over all phase space at the $n^\text{th}$ time step.
$\boldsymbol{\phi}_h^n  \in \real^{p}$  is the vector of multigroup grid functions of the zeroth moment in spatial cells,
where $p=(X+X_f)G$ and $X_f$ is the total number of cell faces in the spatial grid.
 $\mathbf{\bar Q}_h^n\in \real^{d}$ is the vector of corner normalized Planckian source terms.
To generate the  basis for projection of the discrete NBTE
(Eq. \eqref{eq:disc-nbte})
we form the data matrix of snapshots of the normalized radiation intensities for $N$ time steps
ordered chronologically
\begin{equation}
\mathbf{A}\stackrel{\Delta}{=} \big[\boldsymbol{\bar{I}}_h^1\ \boldsymbol{\bar{I}}_h^2\ \dots\ \boldsymbol{\bar{I}}_h^N \big] \, , \quad
\mathbf{A} \in \real^{d \times N} \, .
\end{equation}
The weighted inner product
\begin{equation} \label{w-inner-0}
 \ipp{\boldsymbol{u}}{\boldsymbol{\tilde u}}_W  = \boldsymbol{u}^\top \mathbf{W} \boldsymbol{\tilde u}  \, \quad
 \boldsymbol{u}, \boldsymbol{\tilde u} \in \real^{d} \, ,  \quad \mathbf{W} \in \real^{d \times d}
 \end{equation}
is introduced to perform projection.
 The associated W-norm is given by $\|u\|_W=\sqrt{\ipp{\boldsymbol{u}}{\boldsymbol{u}}_W}$.
The  discrete inner product \eqref{w-inner-0} is  induced by the following one defined for  functions
\begin{equation}
	\ipp{\chi}{\tilde \chi}
 = \int_{0}^{\infty}\int_{4\pi}\int_{\Gamma} \chi(\mathbf{r},\boldsymbol{\Omega},\nu)
 \tilde \chi(\mathbf{r},\boldsymbol{\Omega},\nu)  d^2r d\Omega d\nu
\end{equation}
and derived to  correspond to the space of  discrete  solution and
the discretization scheme \eqref{eq:disc-nbte} of the NBTE \cite{volkwein-2002,volkwein-2013}.
The multigroup approximation with respect to  frequency ($\nu$) leads to
\begin{equation}
\int_{0}^{\infty}\int_{4\pi}\int_{\Gamma} \chi
 \tilde \chi   d^2r d\Omega d\nu = \sum_{g=1}^G  \int_{4\pi}\int_{\Gamma} \chi_g
 \tilde \chi_g   d^2r d\Omega  \, , \quad  \chi_g = \int_{\nu_g}^{\nu_{g+1}} \chi d\nu \, .
\end{equation}
 Then, the discrete-ordinates discretization for directions of particle motion  $\boldsymbol{\Omega}$  gives rise to
\begin{equation}
 \sum_{g=1}^G  \int_{4\pi}\int_{\Gamma} \chi_g \tilde \chi_g   d^2r d\Omega \approx
  \sum_{g=1}^G \sum_{m=1}^M   w_m\int_{\Gamma} \chi_{g,m} \tilde \chi_{g,m}   d^2r  \, ,   \quad
  \chi_{g,m} = \chi_g(\boldsymbol{\Omega}_m) \, .
\end{equation}
The SCB scheme is based on   integration over cell corners and hence we obtain
\begin{multline}
  \sum_{g=1}^G \sum_{m=1}^M   w_m \int_{\Gamma} \chi_{g,m} \tilde \chi_{g,m}   d^2r =
   \sum_{g=1}^G \sum_{m=1}^M   w_m  \sum_{i=1}^X \int_{D_i} \chi_{m,g} \tilde\chi_{m,g}  d^2r \approx \\
   \sum_{g=1}^G \sum_{m=1}^M   w_m  \sum_{i=1}^X  \sum_{\alpha=1}^{C}   \chi_{g,m,i,\alpha} \tilde \chi_{g,m,i,\alpha}S_{i\alpha}  \, .
\end{multline}
As a result, we define the  discrete weighted  inner product
as follows:
\begin{equation} \label{w-inner}
\ipp{\boldsymbol{u}}{\boldsymbol{\tilde u}}_W \stackrel{\Delta}{=} \sum_{g=1}^G\sum_{m=1}^M w_m \sum_{i=1}^{X}\sum_{\alpha=1}^C {u}_{g,m,i,\alpha} {\tilde u}_{g,m,i,\alpha}S_{i,\alpha} \, .
\end{equation}
Thus, the matrix $\mathbf{W}$ is   given by
\begin{equation} \label{eq:Wmat}
	\mathbf{W} \stackrel{\Delta}{=} \bigoplus_{g=1}^{G}\bigoplus_{m=1}^{M}\bigoplus_{i=1}^{X} w_m \mathbf{S},\quad \mathbf{S} =\text{diag}(S_{i,1},\ldots, S_{i,C}).
\end{equation}
Note that in the case of rectangular spatial grids (see Fig. \ref{scb_corner}) we have
\begin{equation}
C=4 \, , \quad S_{i,\alpha} = \frac{1}{4}S_{i,c}
\end{equation}
where $S_{i,c}$ is the area of the $i^{th}$ cell  and
$\mathbf{W}$ is reduced to
\begin{equation} \label{eq:Wmat-rec}
	\mathbf{W} = \frac{1}{4}\bigoplus_{g=1}^{G}\bigoplus_{m=1}^{M}\bigoplus_{i=1}^{X} w_m S_{i,c}\mathbb{I},\quad \mathbb{I}=\text{diag}(1,1,1,1).
\end{equation}

We now construct a weighted data matrix
\begin{equation} \label{hat-A}
\hat{\mathbf{A}} \stackrel{\Delta}{=} \mathbf{W}^{-1/2}\mathbf{A}\mathbf{H}^{-1/2} \, ,
\end{equation}
  where the matrix containing temporal integral weights corresponding to the BE time-integration scheme is $\mathbf{H}\stackrel{\Delta}{=}\text{diag}(\Delta t^1,\ \Delta t^2,\ \dots,\ \Delta t^N)$.
The   thin  singular value decomposition (SVD) of  $\hat{\mathbf{A}}$ yields
\begin{equation}
\hat{\mathbf{A}}=\hat{\mathbf{U}}\hat{\mathbf{S}}\hat{\mathbf{V}}^\top \,  \quad
\hat{\mathbf{U}}\in\real^{d\times r} \, , \quad
\hat{\mathbf{V}}\in\real^{N\times r} \, , \quad
\hat{\mathbf{S}}\in\real^{r\times r} \, ,
\end{equation}
where    $r$ is the rank of $\hat{\mathbf{A}}$,
matrices $\hat{\mathbf{U}}=[\hat{\boldsymbol{u}}_1\ \dots\ \hat{\boldsymbol{u}}_r]$   and  $\hat{\mathbf{V}}=[\hat{\boldsymbol{v}}_1\ \dots\ \hat{\boldsymbol{v}}_r]$ contain
the left and right singular vectors of $\hat{\mathbf{A}}$ as columns, respectively.
$\hat{\mathbf{S}}=\text{diag}(\hat{\sigma_1},\ \hat{\sigma_2},\ \dots,\ \hat{\sigma_r})$ holds the nonzero singular values of $\hat{\mathbf{A}}$ listed in descending value.
 The POD basis   $\{ \boldsymbol{u}_\ell \}_{\ell=1}^r$ is defined by
 \begin{equation} \label{pod-basis}
  \boldsymbol{u}_\ell  =  \mathbf{W}^{1/2}\boldsymbol{\hat u}_\ell  \, .
 \end{equation}
The basis defined via Eq. \eqref{pod-basis} optimally approximates the snapshots
$\{\boldsymbol{\bar{I}}_h^n\}_{n=1}^N$
in the W-norm defined by the weighted inner product in Eq. \eqref{w-inner} \cite{volkwein-2013}.
The normalized intensities are expanded in this basis as
\begin{equation} \label{eq:ibar_exp}
	\boldsymbol{\bar{I}}_h^n = \sum_{\ell=1}^{k} \lambda_\ell^n\boldsymbol{u}_\ell,\quad k\leq r,
\end{equation}

\noindent where $\{\lambda_\ell^n\}_{\ell=1}^k$ is a set of unknown scalar coefficients, which are
 the generalized coordinates for the normalized intensities in the basis $\{\boldsymbol{u}_\ell\}_{\ell=1}^k$.
 	The rank of expansion $k$ is determined by considering the projection of $\hat{\mathbf{A}}$ onto the rank-$k$ weighted POD basis, $\hat{\mathbf{A}}_k=\hat{\mathbf{U}}_k\hat{\mathbf{U}}_k^\top\hat{\mathbf{A}}$, where $\hat{\mathbf{U}}_k=[\hat{\boldsymbol{u}}_1\ \dots\ \hat{\boldsymbol{u}}_k]$.
The errors of this POD approximation in the Frobenius norm and 2-norm are determined by the singular values and given by \cite{ipsen-2009}
\begin{equation}
	\|\hat{\mathbf{A}} - \hat{\mat{U}}_k\hat{\mat{U}}_k^\top\hat{\mathbf{A}}\|_F^2 = \sum_{\ell=k+1}^r\hat{\sigma}_\ell^2,
\end{equation}
and
\begin{gather}
\|\hat{\mathbf{A}} - \hat{\mat{U}}_k\hat{\mat{U}}_k^\top\hat{\mathbf{A}}\|_2^2 = \hat{\sigma}_{k+1}^2.
\end{gather}
The relative error of the POD approximation in the Frobenius norm is therefore
\begin{equation} \label{eq:PODerr_rel}
	\xi^2 = \frac{\|\hat{\mathbf{A}}-\hat{\mathbf{A}}_k\|_F^2}{\|\hat{\mathbf{A}}\|_F^2} = \frac{\sum_{\ell=k+1}^r\hat{\sigma}_\ell^2}{\sum_{\ell=1}^r\hat{\sigma}_\ell^2}.
\end{equation}
The rank of expansion $k$ can thus be determined to enforce the POD approximation error in the relative Frobenius norm to be less than some desired value $\xi\in[0,1]$:
\begin{equation}
	k = \min \bigg\{ j\ :\ \frac{\sum_{\ell=j+1}^r\hat{\sigma}_\ell^2}{\sum_{i=1}^r\hat{\sigma}_\ell^2} \leq \xi^2 \bigg\}   \, . \label{eq:PODerr_k}
\end{equation}

We  substitute the expansion \eqref{eq:ibar_exp} into Eq. \eqref{eq:disc-nbte} to obtain
\begin{equation} \label{eq:nbte_R_exp}
	\mathcal{R}_h^n(\mathbf{U}_k\boldsymbol{\Lambda}^n) = 0
\end{equation}
where
\begin{equation}
 \mathcal{R}_h^n(\mathbf{U}_k\boldsymbol{\Lambda}^n) \stackrel{\Delta}{=}( \mathcal{T}^n_h\mathbf{U}_k\boldsymbol{\Lambda}^n -
 \mathcal{\bar T}_h^n\mathbf{U}_k\boldsymbol{\Lambda}^{n-1} ) +
\mathcal{\bar L}_h\mathbf{U}_k\boldsymbol{\Lambda}^n - \mathbf{\bar Q}_h^n  \, ,
\end{equation}
is the residual of the discretized NBTE,
\noindent $\mathbf{U}_k=[\boldsymbol{u}_1\ \dots\ \boldsymbol{u}_k]$ and $\boldsymbol{\Lambda}^n = (\lambda_1^n\ \dots\ \lambda_k^n)^\top$.
We now project Eq. \eqref{eq:nbte_R_exp}
using the weighted inner product \eqref{w-inner} with
some set of test basis functions $\{\boldsymbol{\psi}_\ell\}_{\ell=1}^k$, creating the $k\times k$ dynamical system whose solution is for
the expansion coefficients  $\{\lambda_\ell^n\}_{\ell=1}^k$
\begin{equation} \label{eq:nbte_R_proj}
	\mathbf{\Psi}^\top\mathbf{W}\mathcal{R}_h^n(\mathbf{U}_k\boldsymbol{\Lambda}^n) = 0,
\end{equation}

\noindent with initial condition
\begin{equation}
	\mathbf{\Psi}^\top\mathbf{W}\mathbf{U}_k\boldsymbol{\Lambda}^0=
\mathbf{\Psi}^\top\mathbf{W}\boldsymbol{\bar{I}}_h^0,
\end{equation}

\noindent where $\mathbf{\Psi}=[\boldsymbol{\psi}_1\ \dots\ \boldsymbol{\psi}_k]$. Just as how the POD basis functions $\{\boldsymbol{u}_\ell\}_{\ell=1}^k$ were found constrained  to yield the  optimal expansion
of $\boldsymbol{\bar{I}}_h$, the test basis functions $\{\boldsymbol{\psi}_\ell\}_{\ell=1}^k$ should yield an optimal projection of the NBTE in the weighted inner product.

Thus, we determine $\mathbf{\Psi}$ such that the solution to Eq. \eqref{eq:nbte_R_proj} satisfies
\begin{equation} \label{eq:pg_min}
		\boldsymbol{\Lambda}^n = \arg\min_{\boldsymbol{\zeta}\in\real^k} \|\mathcal{R}_h^n(\mathbf{U}_k\boldsymbol{\zeta})\|_W^2.
\end{equation}
The  well known least-squares Petrov-Galerkin schemes  \cite{buitanh-willcox-ghattas-2008,carlberg-barone-antil-2017,choi-carlberg-2019} use the test basis functions
\begin{equation} \label{eq:pg_basis}
 \boldsymbol{\psi}_\ell = \frac{d \mathcal{R}_h^n(\lambda_{\ell}^n\boldsymbol{u}_{\ell}) }{d \lambda_{\ell}^n},
\end{equation}

\noindent which can be shown to enforce the condition that
\begin{equation}
 \boldsymbol{\Lambda}^n = \arg\min_{\boldsymbol{\zeta}\in\real^k} \|\mathcal{R}_h^n(\mathbf{U}_k\boldsymbol{\zeta})\|_2^2
\end{equation}
if Eq. \eqref{eq:nbte_R_exp} is projected onto $\mathbf{\Psi}$ in the 2-norm \cite{buitanh-willcox-ghattas-2008,carlberg-barone-antil-2017}. It can be shown that
the same
basis functions are the ones which enforce the condition \eqref{eq:pg_min} on Eq. \eqref{eq:nbte_R_proj}
(see \ref{apdx:test_basis}).
The POD-Petrov-Galerkin projection of the NBTE is thus Eq. \eqref{eq:nbte_R_proj},
where $\mathbf{U}_k$ is
defined by Eq. \eqref{pod-basis}
and $\mathbf{\Psi}$ satisfies Eq. \eqref{eq:pg_basis}.
At every time step, the solution to Eq. \eqref{eq:nbte_R_proj} can be used in the expansion \eqref{eq:ibar_exp} to obtain the normalized intensities over the entire phase space of the problem.

In the   NBTE (Eq. \eqref{eq:disc-nbte}),
 the operators $\mathcal{\bar T}_h^n$ and $\mathcal{\bar L}_h^n$ as well as right-hand side   $\mathbf{\bar Q}_h^n$ depend
 on the  vector  $\boldsymbol{\phi}_h^n \in \mathbb{R}^{p}$ of grid functions of the zeroth moment of  group intensities $\{\boldsymbol{\phi}_{h,g}\}_{g=1}^G$.
 The way the test basis functions $\{\boldsymbol{\psi}_\ell\}_{\ell=1}^k$ are defined leads
  to their dependence on $\boldsymbol{\phi}_h$ as well.
  The ROM for TRT is based on the   POD-Petrov-Galerkin projected NBTE and multilevel system of low-order equations
  \eqref{gloqd-o-form} -  \eqref{grey-meb-o-form}.
   In this ROM the projected NBTE plays the role of the high-order NBTE  in the multilevel system
   \eqref{nbte-o-form} -  \eqref{grey-meb-o-form}.
 The zeroth moment of the group intensity  $\phi_{g}(\boldsymbol{r},t)$ applied
 to normalize the intensity (Eq. \eqref{n-intensity}) and formulate the NBTE (Eq. \eqref{nbte-o-form}) (see also Eq. \eqref{eq:nbte})  is given by $\phi_{g}(\boldsymbol{r},t)=cE_{g}(\boldsymbol{r},t)$ (see Eq. \eqref{phi}), where $E_{g}$ is the solution of the multigroup low-order equations \eqref{gloqd-o-form}.

The multigroup LOQD equations can be approximated in space consistently
with the underlying spatial discretization scheme for the  high-order NBTE by
integration of the discrete  over  discrete angular directions and hence performing corresponding projection in the discrete space.
This leads to algebraically  consistent discretization of the multigroup LOQD equations and BTE \cite{larsen-1984,adams-larsen-2002}.
 If the ROM applies LOQD equations  discretized
  consistently with the   BTE and hence with the NBTE, then
  $\boldsymbol{\phi}_h$  can  be defined directly by the discrete solution of the multigroup moment equations $\boldsymbol{\phi}_{h,g}^{lo}$.
  In this case, $\boldsymbol{\phi}_{h,g}^{lo} =  \boldsymbol{\phi}_{h,g}^{ho} $  on any spatial grid, where
 $\boldsymbol{\phi}_{h,g}^{ho}= \mathcal{K}_h  \boldsymbol{I}_h $ and $\mathcal{K}_h  \cdot \stackrel{\Delta}{=}\sum_{m=1}^M w_m \cdot$.

Another method for approximating the multigroup LOQD equations in space is to apply  independent discretization schemes
\cite{dya-vyag-1993,adams-larsen-2002}.
In this case the zeroth moments from the low-order and high-order equations differ
 by a truncation error and hence
\begin{equation}
\boldsymbol{\phi}_{h,g}^{lo}  -  \boldsymbol{\phi}_{h,g}^{ho}   =  \mathcal{O}(h^\mathfrak{p})  \, ,
\end{equation}
where $\mathfrak{p}$  is determined by the order of accuracy of schemes of LOQD equations and BTE.
The difference tends to zero as $h \to 0$. Due to the difference on a finite spatial grid,
${\phi}_{h,g}^{lo}$ cannot be directly applied for  normalization of the BTE instead of ${\phi}_{h,g}^{ho}$
to obtain a POD expansion of the Eddington tensor based on the expansion of the normalized intensity.
One way to overcome this effect is to formulate a  ROM which  effectively takes into account for this difference.
This approach needs  data on the ratio of cell-wise grid functions of zeroth moments from the low-order and high-order equations.
The developed ROM is based on an alternative way  and uses  directly $\phi_{h,g}^{ho}$. This requires storage of $\boldsymbol{\phi}_h^{n,ho}=
\{\boldsymbol{\phi}_{h,g}^{n,ho}\}_{g=1}^G$  for $\{t^n\}_{n=1}^{N}$.
We form a data matrix of
\begin{equation}
	\mat{\Phi} \stackrel{\Delta}{=}\big [\boldsymbol{\phi}_h^1\ \boldsymbol{\phi}_h^2\ \dots \boldsymbol{\phi}_h^{N} \big] \in \real^{p\times N} \, .
\end{equation}
This data is collected from the same FOM solution used to generate the matrix of normalized intensity snapshots $\mathbf{A}$.
Since the grid functions of $\phi_{h,g}^{ho}$ are of the same dimensionality as grid functions of the LOQD solution,
the additional memory requirements are small in comparison to the storage of basis functions for $\bar{I}_g$.

%
%
\section{Discretization of Moment Equations \label{sec:loqd_disc}}

The system of  low-order equations for moments of the intensity and the MEB equation
(Eqs. \eqref{gp-loqd}, \eqref{gqd}, and \eqref{grey_meb}) are
 approximated by the BE time integration method.  In this study we consider TRT problems
in 2D Cartesian geometry on rectangular grids. The multilgroup LOQD equations are discretized in space with the 2$^\text{nd}$-order finite volume  (FV)  method \cite{aristova-vya-avk-m&c1999,pg-dya-jcp-2020}.
The FV method is formulated for the cell-average and face-average  group radiation energy density
and the  face-average normal components of the group radiation flux.
Fig. \ref{fig: fv-uknowns} shows the grid functions of unknowns of the FV scheme in the $i^\text{th}$ spatial cell.
The  multigroup balance equation \eqref{gp_ebal} is integrated over the spatial cell $\mathcal{D}_i$.
The multigroup first moment equation \eqref{gp_mbal}  is
 integrated over   bottom (B), left (L), right (R), and  top (T)  halves of $\mathcal{D}_i$.
The formulation of this FV scheme specifies the grid functions of
 the Eddington tensor as  scheme coefficients. They are the cell-average values of diagonal elements of tensor
 $\mathfrak{f}_{\alpha\alpha,g}$   ($\alpha =x,y)$,  face-average values of $\mathfrak{f}_{xy,g}$ on all faces,
 face-average values of $\mathfrak{f}_{xx,g}$ on the left and right faces, and  face-average values of  $\mathfrak{f}_{yy,g}$ on the top and bottom faces.
 Fig. \ref{fig: fv-ET} illustrates  the grid functions of  the Eddington tensor on the FV scheme in $\mathcal{D}_i$.
  The temperature is defined by its cell-average value. The spatial discetization scheme for the
  grey LOQD equations is  algebraically  consistent  with the FV scheme for the multigroup LOQD equations \cite{pg-dya-jcp-2020}.
  It is derived by summing the discrete multigroup LOQD equations over groups. This scheme
  is formulated  for the same kind of grid functions of cell-wise unknowns as the FV scheme for the multigroup LOQD equations (see Fig. \ref{fig: fv-scheme}) .

\begin{figure}[ht!]
	\vspace{-.5cm}
	\centering
	\subfloat[unknowns in $\mathcal{D}_i$ \label{fig: fv-uknowns}]
{\begin{picture}(160,120)
\drawline(20,20)(140,20)(140,100)(20,100)(20,20)
\put(80,60){\circle*{4}}
      \put(70,67.5){{\footnotesize $
      E_{g,i,c}$}}
\put(140,60){\circle*{4}}
      \put(114,67.5){{\footnotesize $E_{g,i,R}$}}
      \put(107,52.5){{\footnotesize $F_{x,g,i,R}$}}
\put(20,60){\circle*{4}}
       \put(23,67.5){{\footnotesize $E_{g,i,L}$}}
     \put(23,52.5){{\footnotesize $F_{x,g,i,L}$}}
\put(80,20){\circle*{4}}
      \put(55,25){{\footnotesize $E_{g,i,B}$}}
       \put(83,25){{\footnotesize $F_{y,g,i,B}$}}
\put(80,100){\circle*{4}}
      \put(55,90){{\footnotesize 
      $E_{g,i,T}
      $}}
       \put(83,90){{\footnotesize $F_{y,g,i,T}$}}
\end{picture}}
\hspace{2cm}
	\subfloat[grid functions of    $\mathfrak{f}_{g}$  \label{fig: fv-ET}]
{\begin{picture}(160,120)
\drawline(20,20)(140,20)(140,100)(20,100)(20,20)
\put(80,60){\circle*{4}}
      \put(66,67.5){{\footnotesize $\mathfrak{f}_{xx,g,i,c}$}}
      \put(66,52.5){{\footnotesize $\mathfrak{f}_{yy,g,i,c}$}}
\put(140,60){\circle*{4}}
      \put(105,67.5){{\footnotesize $\mathfrak{f}_{xx,g,i,R}$}}
      \put(105,52.5){{\footnotesize $\mathfrak{f}_{xy,g,i,R}$}}
\put(20,60){\circle*{4}}
       \put(23,67.5){{\footnotesize $\mathfrak{f}_{xx,g,i,L}$}}
     \put(23,52.5){{\footnotesize $\mathfrak{f}_{xy,g,i,L}$}}
\put(80,20){\circle*{4}}
      \put(45,25){{\footnotesize $\mathfrak{f}_{yy,g,i,B}$}}
       \put(85,25){{\footnotesize $\mathfrak{f}_{xy,g,i,B}$}}
\put(80,100){\circle*{4}}
      \put(45,90){{\footnotesize $\mathfrak{f}_{yy,g,i,T}$}}
       \put(85,90){{\footnotesize $\mathfrak{f}_{xy,g,i,T}$}}
\end{picture}}
\vspace{0.5cm}
	\caption{Unknowns of the FV scheme and grid functions of the Eddington tensor in $\mathcal{D}_i$	\label{fig: fv-scheme} }
\end{figure}
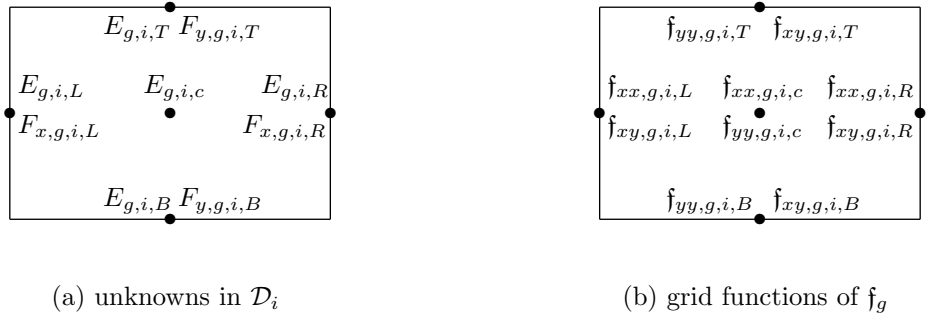

%
%
\section{Formulation of the Reduced Order Model} \label{sec:rom}

The ROM for TRT for the given low-rank $k$ is defined by the hierarchy of equations defined as follows
\begin{itemize}
\item The POD-Petrov-Galerkin projected NBTE  for
 the expansion coefficients  $\{\lambda_\ell^n\}_{\ell=1}^k$
\begin{subequations} \label{podpg-nbte}
\begin{equation}
	\mathbf{\Psi}^\top\mathbf{W}\mathcal{R}_h^n(\mathbf{U}_k\boldsymbol{\Lambda}^n) = 0,
\end{equation}
\begin{equation} \mathbf{\Psi}^\top\mathbf{W}\mathbf{U}_k\boldsymbol{\Lambda}^0=\mathbf{\Psi}^\top\mathbf{W}\boldsymbol{\bar{I}}_h^0,
\end{equation}
\end{subequations}
\item the expansion of  the grid functions of the Eddington tensor
$\boldsymbol{\mathfrak{f}}_h^n = \{\boldsymbol{\mathfrak{f}}_{h,g}^n \}_{g=1}^G$
 and the boundary factors
$\boldsymbol{\beta}_h^n = \{\boldsymbol{\beta}_{h,g}^n \}_{g=1}^G$
 with the angular moments of the POD basis functions
\begin{equation} \label{et-pod}
\boldsymbol{\mathfrak{f}}_h^n = \sum_{\ell=1}^k \lambda_k^n{\boldsymbol{\mathfrak{h}}}_{\ell} \, ,
\end{equation}
\begin{equation}  \label{bf-pod}
\boldsymbol{\beta}_h^n = \sum_{\ell=1}^k \lambda_k^n{\boldsymbol{b}}_{\ell} \, ,
\end{equation}
where
\begin{equation}
 \boldsymbol{\mathfrak{h}}_\ell=\mathcal{H}_h\boldsymbol{u}_{\ell} \, ,
\quad
 \boldsymbol{b}_\ell=\mathcal{B}_h\boldsymbol{u}_{\ell} \, ,
\end{equation}
and
\begin{equation}
\mathcal{H}_h \cdot \stackrel{\Delta}{=}  \sum_{m=1}^M w_m \boldsymbol{\Omega}_m\otimes\boldsymbol{\Omega}_m \cdot \, , \quad
\mathcal{B}_h \cdot \stackrel{\Delta}{=}
  \sum_{m=1}^M  w_m  |\boldsymbol{n}_\Gamma\cdot\boldsymbol{\Omega}_m|\ \cdot \, ,
\end{equation}
are  discrete operators of corresponing angular moments in discrete space (see Eq. \eqref{H-oper} and  and \eqref{beta-oper}),

\item the  discretized multigroup LOQD equations for grid functions of moments $\boldsymbol{\varphi}_{h,g}$
\begin{equation} \label{mloqd-o-h}
 \frac{1}{\Delta t^n}\big(  \boldsymbol{\varphi}_{h,g}^n -  \boldsymbol{\varphi}_{h,g}^{n-1} \big) +
  \mathcal{M}_{h,g}^n \boldsymbol{\varphi}_{h,g}^n = \mathbf{q}_{h,g}^n \, , \
 \ g \in \mathbb{N}(G) \, ,
\end{equation}
\item  the discretized  effective grey  LOQD equations  for grids functions of   moments $\boldsymbol{\hat \varphi}_{h}$
\begin{equation} \label{gloqd-o-h}
 \frac{1}{\Delta t^n}\big(  \boldsymbol{\hat \varphi}_h^n -  \boldsymbol{\hat \varphi}_h^{n-1} \big)  +
  \mathcal{\hat M}_h^n \boldsymbol{\hat \varphi}_h^n = \mathbf{\hat q}_h^n \, ,
  \end{equation}
\item the   discretized effective grey MEB  equation    for $\boldsymbol{\hat \varphi}_{h}$ and grids functions of   $T$
\begin{equation}\label{eb-grey}
\frac{1}{\Delta t^n} \big(  \varepsilon_h(T^n) - \varepsilon_h(T^{n-1}) \big)=    \mathcal{P}_h^n
  \boldsymbol{\hat \varphi}_{h}^n +   \tilde q_h^n  \, .
\end{equation}
\end{itemize}

The online stage of the ROM for TRT problems is outlined in Algorithm \ref{alg:rom_alg}.
The system of equations of the ROM is solved iteratively.
On every time step, the algorithm is defined by two main iteration cycles:
  outer and inner iterations.
The outer iteration cycle updates
the expansion coefficients $\{\lambda^n\}_{\ell=1}^k$  with the estimate of temperature $T^{n{(s)}}$.
The operator of the NBTE depends on temperature as well as  the test basis functions $\{\boldsymbol{\psi}_\ell\}_{\ell=1}^k$.
These elements are updated to define the problem for the expansion coefficients.
Computing the solution to the projected NBTE  (Eq. \eqref{podpg-nbte}) requires inverting
a $k\times k$ dense linear system which can be solved at low-cost for small $k$.
The grid functions of the
Eddington tensor (Eq. \eqref{et-pod})
 and boundary factors (Eq. \eqref{bf-pod})  are  linear combinations of the POD basis functions
 of the normalized intensity with updated expansion coefficients.
 These data provide updated closures for the multigroup LOQD equations.
  A maximum number of
  outer iterations in the ROM algorithm is defined as $s_\text{max}$ which is used to further reduce the ROM's computational cost.
The inner iterations involve solution of the multilevel moment equations with
the $s^\text{th}$ iterative estimate of Eddington tensor and boundary factors.
The solution of the multigroup LOQD equations generates the multigroup
 functions to compute effective grey opacities and Eddington tensor
 which are coefficients of the effective grey LOQD equations.
 The  estimate of temperature is obtained by solving the
 coupled system of MEB and grey LOQD equations.
 Note that at the beginning of each time step, $\boldsymbol{\mathfrak{f}}_g$ and $\beta_g$ are initialized with the previous time step solution (or initial condition).
 The system of low-order equations is solved with    these initialized quantities to obtain the first estimate
temperature on the new time step.

\begin{algorithm}[h]
{\small
\DontPrintSemicolon
Input: $k$, $\mathbf{U}_k$,
$\{\boldsymbol{\mathfrak{h}}_{\ell} \}_{\ell=1}^k$,
$\{\boldsymbol{b}_{\ell} \}_{\ell=1}^k$,
$\mat{\Phi}$, $s_{max}$\;
$n=0$\;
\While{$t^n \leq t^\text{end}$}{
$s=-1$, $T^{n(0)}=T^{n-1}$,
$\boldsymbol{\mathfrak{f}}_h^{n(0)}=\boldsymbol{\mathfrak{f}}_h^{n-1}$, $\boldsymbol{\beta}_{h}^{n(0)}=\boldsymbol{\beta}_{h}^{n-1}$\;
\While{$s<s_{max}  \! \land  \! \Big(  \!|| T^{n(s)} - T^{n(s-1)} ||>\epsilon_s||T^{n(s)}|| \! \lor  \!
||E^{n(s)} - E^{n(s-1)}||> \epsilon_s||E^{n(s)}||  \! \Big)$ }{
$\bullet$ outer iterations\;
 $s=s+1$ \;
 \If{$s>0$}{
$T^{n(s)} \, \leadsto \, \mathcal{R}_{h}^{n(s)}$, $\mathbf{\Psi}^{n(s)}$ \;
 Solve:  $\big(\mathbf{\Psi}^{n(s)}\big)^\top\mathbf{W}\mathcal{R}_h^{n\, (s)} (\mathbf{U}_k\boldsymbol{\Lambda}^{n\, (s)}) = 0$\;
 $ \boldsymbol{\Lambda}^{n (s)}\  \leadsto  \
 \boldsymbol{\mathfrak{f}}_h^{n \,(s)},
   \boldsymbol{\beta}_{h}^{n\,(s)}$   \;
   }
$p=-1$,  $T^{n(0,s)}=T^{n(s)} $ \;
\While{$|| T^{n(p,s)} - T^{n(p-1,s)} ||> \epsilon_{p}||T^{n(p,s)}||  \, \lor \,
 || E^{n(p,s)} - E^{n(p-1,s)} ||>\ \epsilon_{p}||E^{n(p,s)}||$ }{
 $\bullet$ inner iterations\;
 $p=p+1$ \;
$ T^{n(p,s)},   \boldsymbol{\mathfrak{f}}_{h,g}^{n \,(s)},
   \boldsymbol{\beta}_{h,g}^{n\,(s)}  \ \leadsto  \ \mathcal{M}_{h,g}^{n (p,s)} \, , \mathbf{q}_{h,g}^{n   (p,s)}$ \;
   Solve:  $ (\Delta t^n)^{-1}\big(  \boldsymbol{\varphi}_{h,g}^{n   (p,s)} -  \boldsymbol{\varphi}_{h,g}^{n-1} \big) +
  \mathcal{M}_{h,g}^{n   (p,s)} \boldsymbol{\varphi}_{h,g}^{n   (p,s)} = \mathbf{q}_{h,g}^{n   (p,s)}$\;
 $\boldsymbol{\varphi}_{h,g}^{n  (p,s)},
 \boldsymbol{\mathfrak{f}}_{h,g}^{n \,(s)},
   \boldsymbol{\beta}_{h,g}^{n\,(s)} \leadsto \mathcal{\hat M}_h^{n (p,s)}$\;
 Solve: \;
\begin{fleqn}
\begin{align*}
\begin{cases}
    (\Delta t^n)^{-1}   \big(  \boldsymbol{\hat \varphi}_{h}^{n  (p+1,s)} -  \boldsymbol{\hat \varphi}_{h}^{n-1} \big)+
 \mathcal{\hat M}_h^{n (p,s)}[T^{n  (p+1,s)}] \boldsymbol{\hat \varphi}_{h}^{n   (p+1,s)}=
\mathbf{\hat q}_h[T^{n  (p+1,s)}]   \\
 (\Delta t^n)^{-1}     \big(  \varepsilon_h(T^{n   (p+1,s)}) - \varepsilon_h(T^{n-1}) \big)=
  \mathcal{P}_h[T^{n  (p+1,s)}] \boldsymbol{\hat \varphi}_{h}^{n  (p+1,s)}+   \tilde q_h[T^{n  (p+1,s)}]
  \end{cases}
  \end{align*}
\end{fleqn}
}
$T^{n \, (s+1)} \leftarrow  T^{n \, (p+1,s)}$ \;
}
$T^{n} \leftarrow  T^{n \, (s+1)}$\;
$\boldsymbol{\mathfrak{f}}_{h}^n \leftarrow \boldsymbol{\mathfrak{f}}_{h}^{n \, (s)}$,
 $ \boldsymbol{\beta}_{h}^{n}\leftarrow  \boldsymbol{\beta}_{h}^{n\,(s)}$\;
}
}
\caption{ROM algorithm \label{alg:rom_alg}}
\end{algorithm}

At the offline stage, the FOM is solved to generate the data for the snapshot matrix  $\mat{A}$ and $\mat{\Phi}$.
 Algorithm \ref{alg:fom_alg} presents the  method for solving  FOM based on the BTE.
The discretization scheme of the NBTE in the ROM is algebraically consistent  with the SCB scheme for the BTE.
The solution of the FOM provides snapshots of intensities  $\{\boldsymbol{I}_h^n\}_{n=1}^N$.
Then this numerical solution is normalized according to  Eq. \eqref{eq:nbte_gfunc}.
The offline stage is summarized in Algorithm \ref{alg:offline_alg}.
Note that since the basis functions are known,   $\boldsymbol{\mathfrak{h}}_{\ell}$ and    $\boldsymbol{b}_{\ell}$ as  the integrals of $\boldsymbol{u}_{\ell}$   can be handled in the offline stage.

\begin{algorithm}[h]
{\small
\DontPrintSemicolon
$n=0$\;
\While{$t^n \leq t^\text{end}$}{
$s=-1$, $T^{n(0)}=T^{n-1}$,
$\boldsymbol{\mathfrak{f}}_h^{n(0)}=\boldsymbol{\mathfrak{f}}_h^{n-1}$, $\boldsymbol{\beta}_{h}^{n(0)}=\boldsymbol{\beta}_{h}^{n-1}$\;
\While{$s<s_{max}  \! \land  \! \Big(  \!|| T^{n(s)} - T^{n(s-1)} ||>\epsilon_s||T^{n(s)}|| \! \lor  \!
||E^{n(s)} - E^{n(s-1)}||> \epsilon_s||E^{n(s)}||  \! \Big)$ }{
$\bullet$ outer iterations\;
 $s=s+1$ \;
 \If{$s>0$}{
$T^{n(s)} \, \leadsto  \, \mathcal{L}_{h}^{n(s)}, \mathbf{Q}_h^{n(s)}$\;
  Solve:  $  \mathcal{T}_h^n \big(\boldsymbol{I}_h^{n(s)} -   \boldsymbol{I}_h^{n-1} \big)
  + \mathcal{L}_h^{n(s)} \boldsymbol{I}_h^{n(s)}  = \mathbf{Q}_h^{n(s)}$\;
 $ \boldsymbol{I}_h^{n(s)}\  \leadsto \
\boldsymbol{\mathfrak{f}}_h^{n \,(s)},
   \boldsymbol{\beta}_{h}^{n\,(s)}$   \;
 }
$p=-1$,  $T^{n(0,s)}=T^{n(s)} $ \;
\While{$|| T^{n(p,s)} - T^{n(p-1,s)} ||> \epsilon_{p}||T^{n(p,s)}||  \, \lor \,
 || E^{n(p,s)} - E^{n(p-1,s)} ||> \epsilon_{p}||E^{n(p,s)}||$ }{
$\vdots$
}
$T^{n \, (s+1)} \leftarrow  T^{n \, (p+1,s)}$ \;
}
$T^{n} \leftarrow  T^{n \, (s+1)}$\;
$\boldsymbol{\mathfrak{f}}_{h}^n \leftarrow \boldsymbol{\mathfrak{f}}_{h}^{n \, (s)}$,
 $ \boldsymbol{\beta}_{h}^{n}\leftarrow  \boldsymbol{\beta}_{h}^{n\,(s)}$\;
}
}
\caption{FOM algorithm \label{alg:fom_alg}}
\end{algorithm}

\begin{algorithm}[h]
	\begin{enumerate}
		\item solve FOM to compute $\{\boldsymbol{I}_h^n\}_{n=1}^N$ and $\{\boldsymbol{\phi}_h^n\}_{n=1}^N$
		\item normalize the FOM solution to obtain  $\{\boldsymbol{\bar{I}}_h^n\}_{n=1}^N$
		\item form $\mat{A}$ and  $\mat{\Phi}$
		\item form the weighted data  matrix $\hat{\mathbf{A}} = \mathbf{W}^{-1/2}\mathbf{A}\mathbf{H}^{-1/2}$
		\item compute the thin SVD   $\hat{\mathbf{A}}=\hat{\mathbf{U}}\hat{\mathbf{S}}\hat{\mathbf{V}}^\top$
		\item determine the basis dimension $	k = \min \bigg\{ j\ :\ \frac{\sum_{\ell=j+1}^r\hat{\sigma}_\ell^2}{\sum_{i=1}^r\hat{\sigma}_\ell^2} \leq \xi^2 \bigg\}$
		\item truncate weighted POD basis:  $\hat{\mathbf{U}}_k=[\hat{\boldsymbol{u}}_1\ \dots\ \hat{\boldsymbol{u}}_k]$
		\item compute the low-rank POD basis:  $  \mathbf{U}_k  = \mathbf{W}^{1/2} \hat{\mathbf{U}}_k$
		\item compute $\{\boldsymbol{\mathfrak{h}}_{\ell} \}_{\ell=1}^k$,
			$\{\boldsymbol{b}_{\ell} \}_{\ell=1}^k$
	\end{enumerate}
	\caption{The offline stage of ROM \label{alg:offline_alg}}
\end{algorithm}

\clearpage
%
%
\section{Numerical Results} \label{sec:numerical_results} \label{sec:results}

%
\subsection{Test Problem}
To analyze the properties of the proposed ROM, we use a 2-dimensional extension of the well-known Fleck-Cummings (F-C) test problem \cite{fleck-1971}. This F-C test takes the form of a square homogeneous domain in the $x-y$ plane, 6 cm in length on both sides. The domain is initially at a temperature of $T^0$, the left boundary of the domain is subject to incoming radiation with blackbody spectrum at a temperature of $T^\text{in}$, and there is no  incoming radiation at  other boundaries.
The selected problem parameters to be used here are $T^0 = 1\ \text{eV}$ and $T^\text{in}=1\ \text{KeV}$. The material is characterized by an opacity of
\begin{equation} \label{eq:fc-opac}
	\varkappa(\nu, T) = \frac{27}{\nu^3}\bigg(1-e^{-\frac{\nu}{T}}\bigg),
\end{equation}

\noindent with $\nu$ and $T$ measured in KeV, and a material energy density that is linear in temperature
\begin{equation}
	\varepsilon(T) = c_vT,
\end{equation}

\noindent
with material heat capacity $c_v = 0.5917 a_R (T^\text{in})^3$.
The time interval of the problem is  $t\in[0,3\ \text{ns}]$.

\begin{table}[ht!]
	\centering
	\caption{Upper boundaries for each frequency group}
	\begin{tabular}{|l|l|l|l|l|l|l|l|l|l|}
		\hline
		$g$ &1&2&3&4&5&6&7&8&9 \\ \hline
		$\nu_{g}$ [KeV]
		& 0.7075
		& 1.415
		& 2.123
		& 2.830
		& 3.538
		& 4.245
		& 5.129
		& 6.014
		& 6.898 \\ \hline\hline
		$g$&10&11&12&13&14&15&16&17& \\ \hline
		$\nu_{g}$ [KeV]
		& 7.783
		& 8.667
		& 9.551
		& 10.44
		& 11.32
		& 12.20
		& 13.09
		& 1$\times 10^{7}$ & \\ \hline
	\end{tabular}
	\label{tab:freq_grps}
\end{table}

\begin{figure}[ht!]
	\centering
	\begin{tabular}{c|c|c|c}
		& t=1ns & t=2ns & t=3ns \\ \hline &&& \\[-.3cm]
		$T$ &
		\raisebox{-.5\height}{\includegraphics[trim=1cm 0cm 4cm .5cm,clip,height=3.5cm]{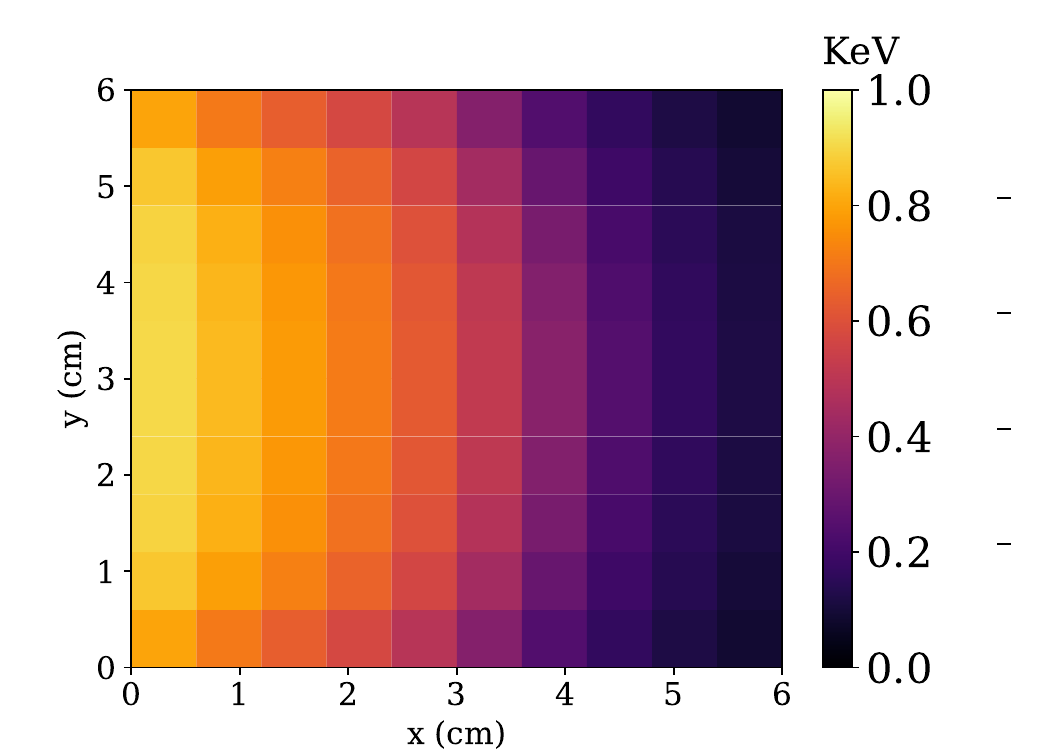}}&
		\raisebox{-.5\height}{\includegraphics[trim=1cm 0cm 4cm .5cm,clip,height=3.5cm]{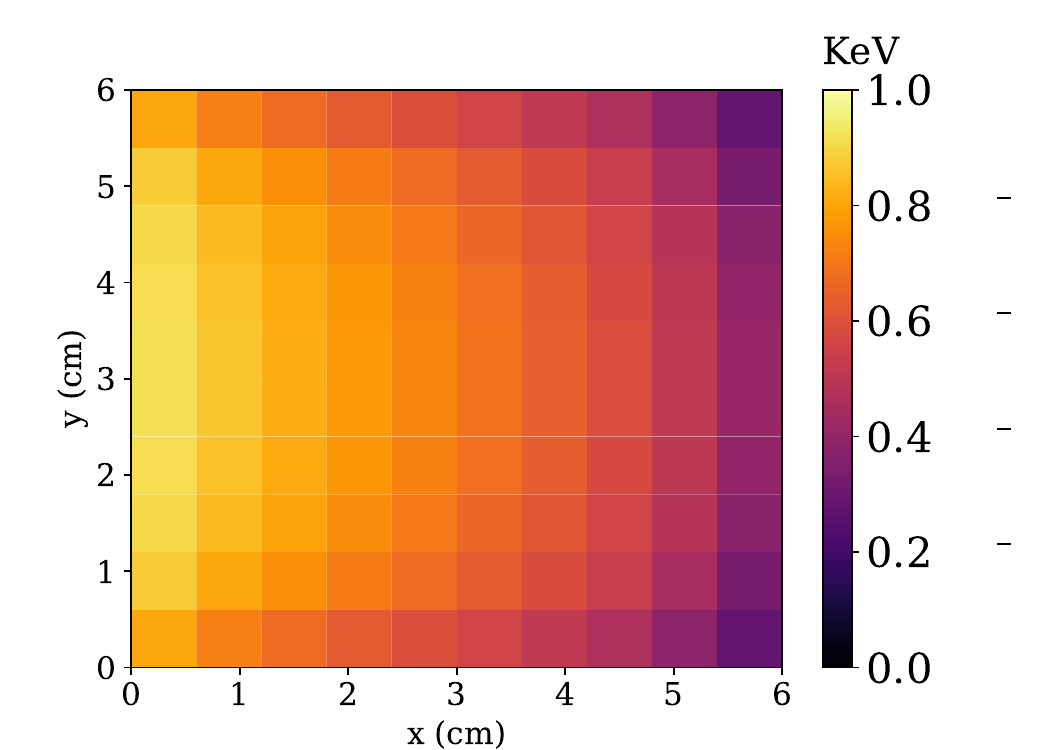}}&
		\raisebox{-.5\height}{\includegraphics[trim=1cm 0cm 0cm .5cm,clip,height=3.5cm]{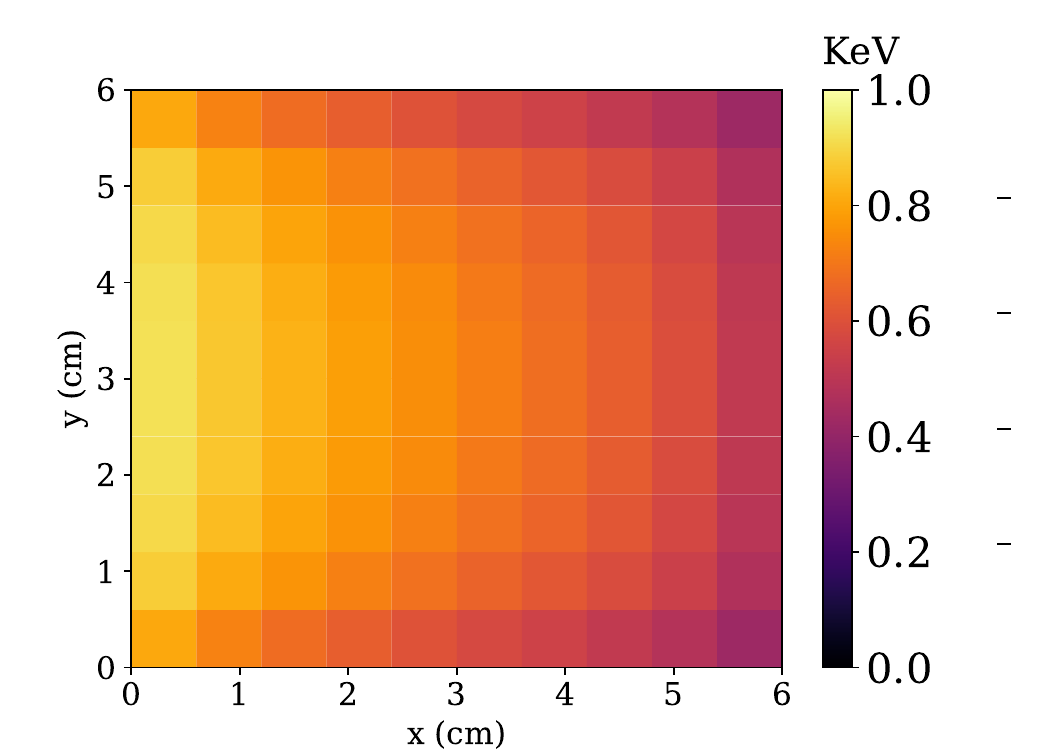}}\\[.1cm] \hline &&& \\[-.3cm]
		$E$ &
		\raisebox{-.5\height}{\includegraphics[trim=1cm 0cm 4cm 0cm,clip,height=3.65cm]{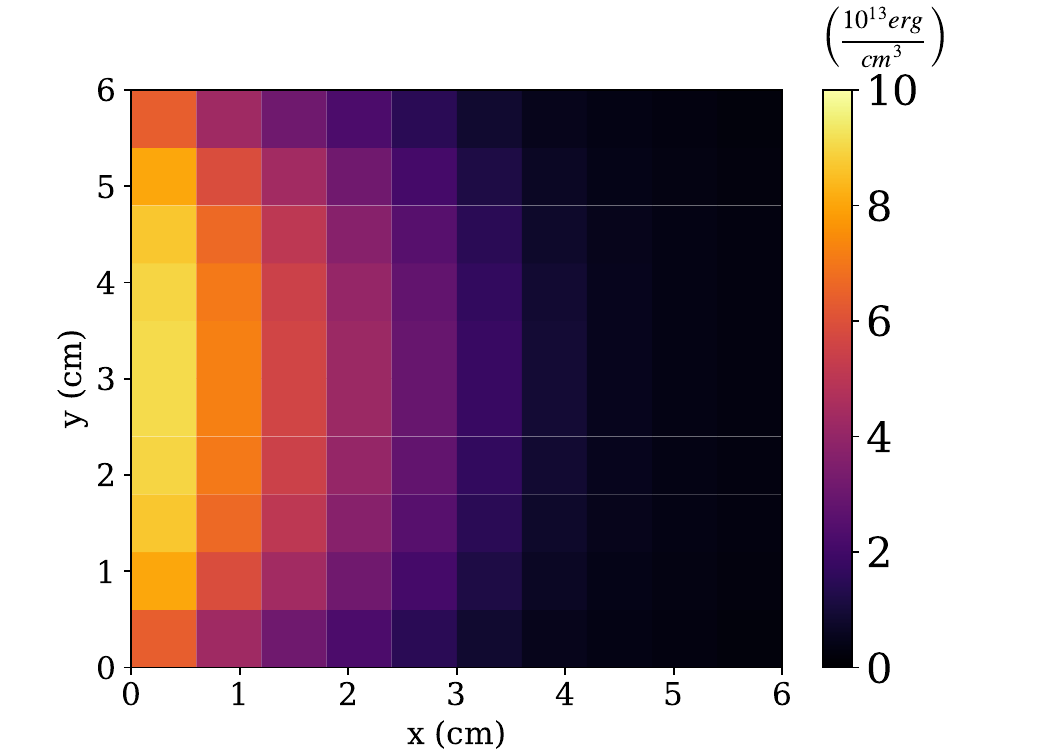}}&
		\raisebox{-.5\height}{\includegraphics[trim=1cm 0cm 4cm 0cm,clip,height=3.65cm]{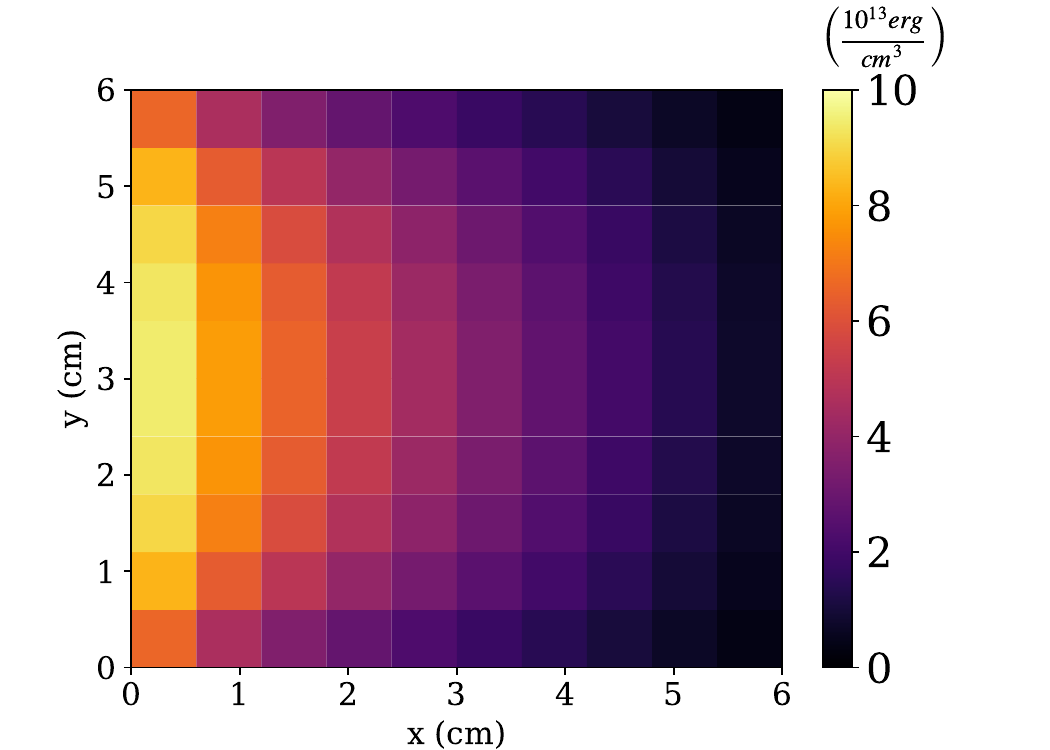}}&
		\raisebox{-.5\height}{\includegraphics[trim=1cm 0cm 0cm 0cm,clip,height=3.65cm]{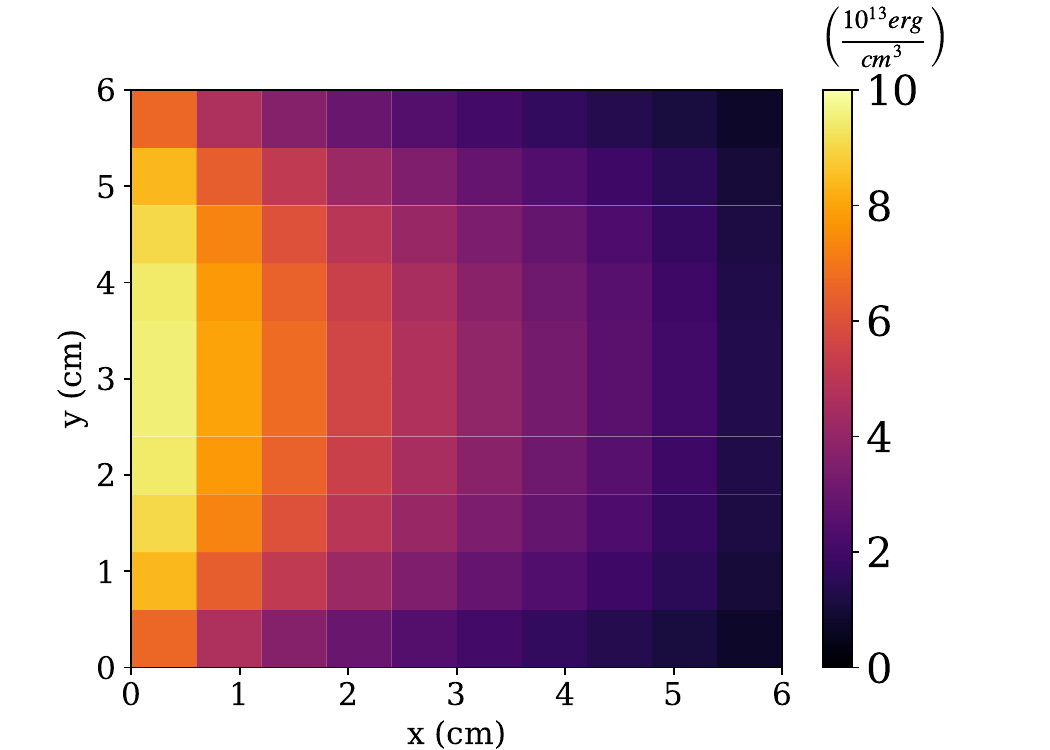}}
	\end{tabular}
	\vspace*{.25cm}
	\caption{F-C test solution for the material temperature $(T)$ and total radiation energy density $(E)$ over the spatial domain at times t=1,\,2,\,3\,ns.}
	
	\label{fig:fom_solution}
\end{figure}

A uniform grid of  $10\times 10$ cells (i.e. $X=100$) with side lengths of $\Delta x = \Delta y = 0.6\ \text{cm}$  is used to discretize the slab. $G=17$ frequency groups are defined as shown in Table \ref{tab:freq_grps}.
The Abu-Shumays angular quadrature set q461214 with
36 discrete directions per quadrant  is used  \cite{abu-shumays-2001}. The total number of angular directions is $M=144$.
The F-C problem is solved for the
specified time interval
with $N=150$ uniform time steps $\Delta t = 2\times 10^{-2} \ \text{ns}$.
When generating ROM solutions to the F-C problem, the following convergence criteria are used
(ref. Algorithm \ref{alg:rom_alg}):
$\epsilon_s=\epsilon_{p}=10^{-14}$.

The solution to this F-C test for the material temperature and total radiation energy density at times
$t=1,2,3$ ns is depicted in Fig. \ref{fig:fom_solution}. The solutions of both $T$ and $E$ take the form of a wave that first rapidly forms on the left boundary before propagating to the right. After this the domain is continuously heated.
Eventually the solution reaches a regime close to steady state. Fig. \ref{fig:fc_opac_vs-nu} plots the F-C opacity (Eq. \eqref{eq:fc-opac}) against photon energy for several material temperatures. The considered ranges for $\nu$ and $T$ are typical for the F-C test.

\begin{figure}[ht!]
	\centering
	\includegraphics[width=0.5\textwidth]{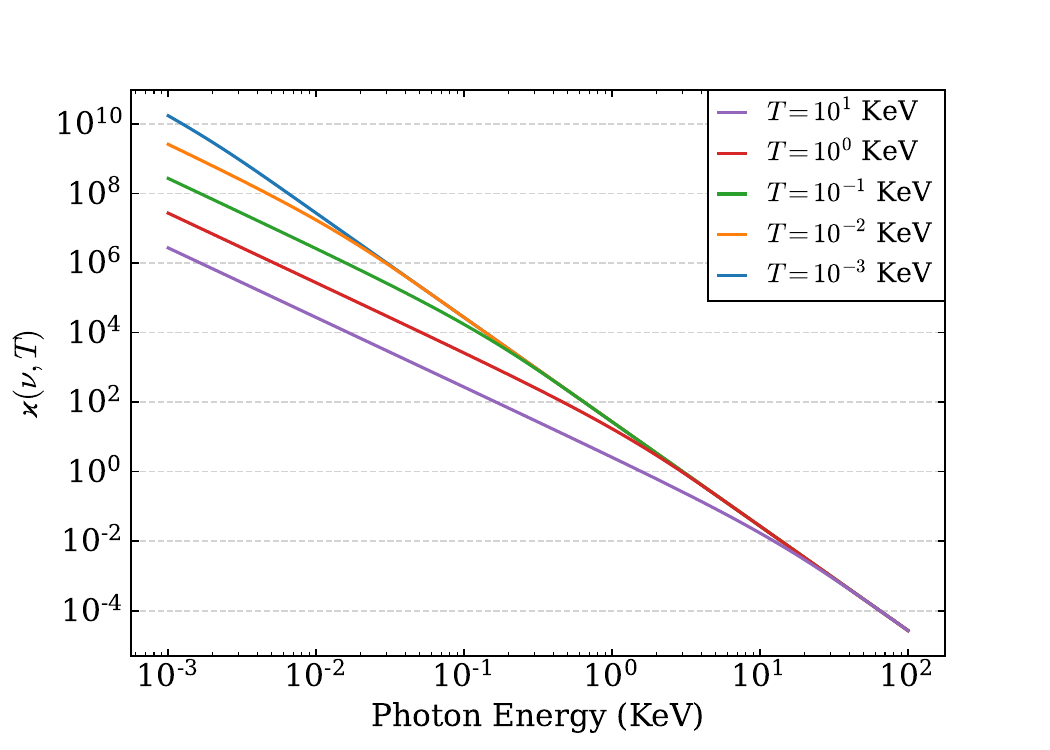}
	\caption{The F-C test spectral opacity $\varkappa(\nu,T)$ (ref. Eq. \eqref{eq:fc-opac}) plotted vs photon frequency for several material temperatures. \label{fig:fc_opac_vs-nu} }
\end{figure}

%
\subsection{Analysis of POD Databases \& Singular Values}
The singular value distributions of the weighted snapshot matrices ($\hat{\mathbf{A}}$) holding the FOM solution for $\bar{I}$ are shown in Fig. \ref{fig:Ibar_Svals}. Singular values for FOM solutions found with both $s_\text{max}=\infty$ and $s_\text{max}=1$ are shown. The difference in singular values in these two cases is insignificant.
Note that the outer iterations converge fast and   the iterative solution obtained with a single transport sweep in case of $s_\text{max}=1$ well approximates the problem solution \cite{gol'din-1964,gol'din-CMMP-1966,dya-jcp-2019}.
In both plots there are three discernible regions of sharp decrease in singular value magnitudes, each followed by a region where the rate of decrease slows dramatically. The third and last region of stagnation is where the limits of numerical precision become dominant and no further reduction is possible. Table \ref{tab:Ibar_ranks} displays the ranks of expansion $k$ that correspond to several different values of $\xi$, calculated via 
Eq. \eqref{eq:PODerr_k}. The values shown are calculated for the FOM solution using $s_\text{max}=\infty$. The ranks for the solution using $s_\text{max}=1$ are identical and thus not shown.

\begin{figure}[ht!]
	\vspace{-.5cm}
	\centering
	\subfloat[$s_\text{max}=\infty$]{\includegraphics[width=.5\textwidth]{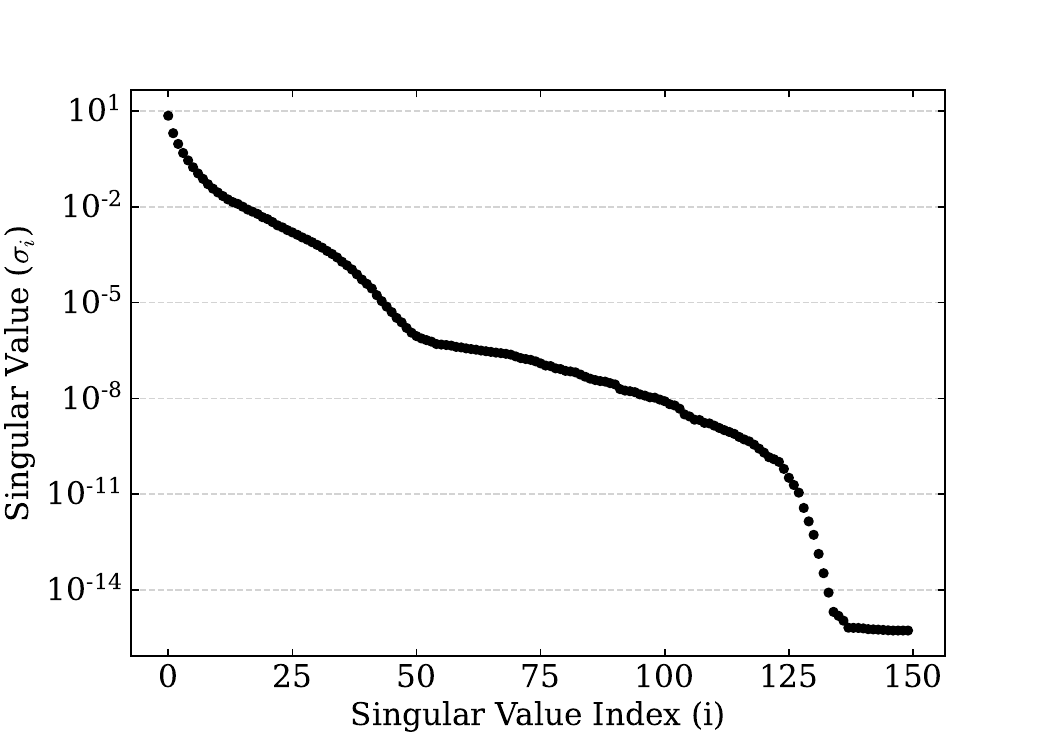}}
	\subfloat[$s_\text{max}=1$]{\includegraphics[width=.5\textwidth]{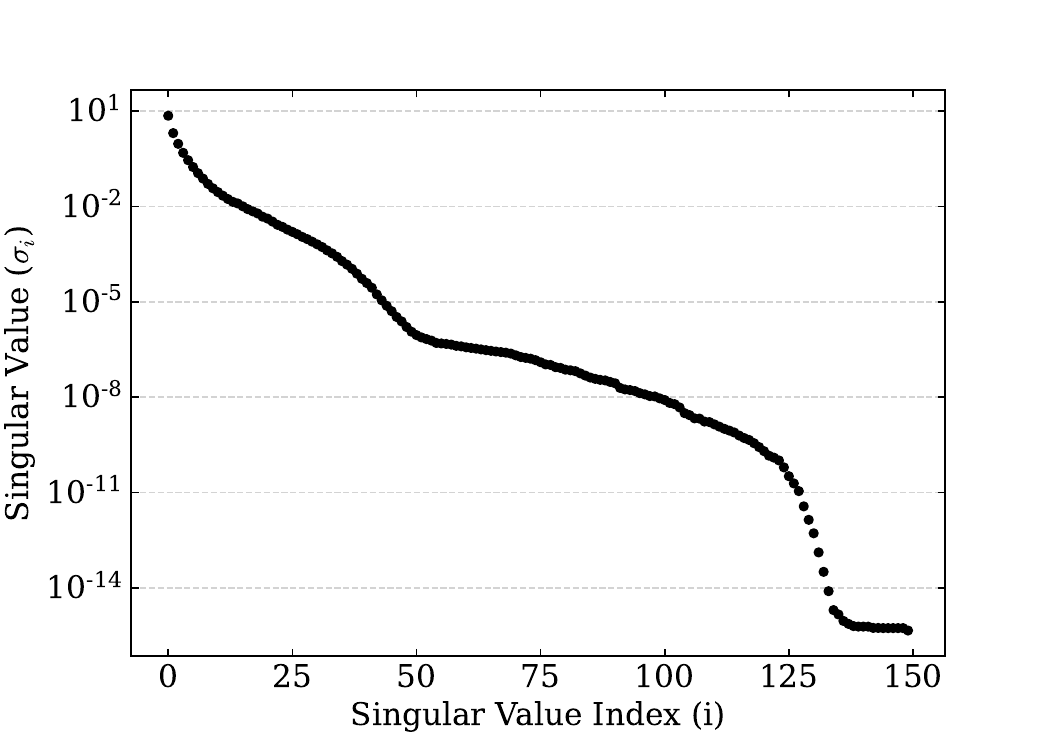}}
	\caption{Singular value spectra of the weighted snapshot matrices ($\hat{\mathbf{A}}$) holding $\bar{I}$ obtained by the FOM with $s_\text{max}=\infty$ and $s_\text{max}=1$ }
	\label{fig:Ibar_Svals}
\end{figure}

\begin{figure}[ht!]
	\centering
	\def\arraystretch{1.05}
	\begin{tabular}{|l|l|l|l|l|l|l|l|l|}
	\hline
	$\xi$ & $10^{-2}$ & 	$10^{-4}$ & 	$10^{-6}$ & 	$10^{-8}$ & 	$10^{-10}$ & 	$10^{-12}$ & 	$10^{-14}$ & 	$10^{-16}$	 \\ \hline
     $k$ &  9 & 	32	 & 45	 & 88	 & 117	 & 128	 & 132	 & 149	\\ \hline
	\end{tabular}
	\captionof{table}{Ranks $k$ for the POD expansion of $\bar{I}$ corresponding to different values of $\xi$
		\label{tab:Ibar_ranks}}
\end{figure}

\begin{figure}[ht!]
	\vspace{-.5cm}
	\centering
	\includegraphics[width=.5\textwidth]{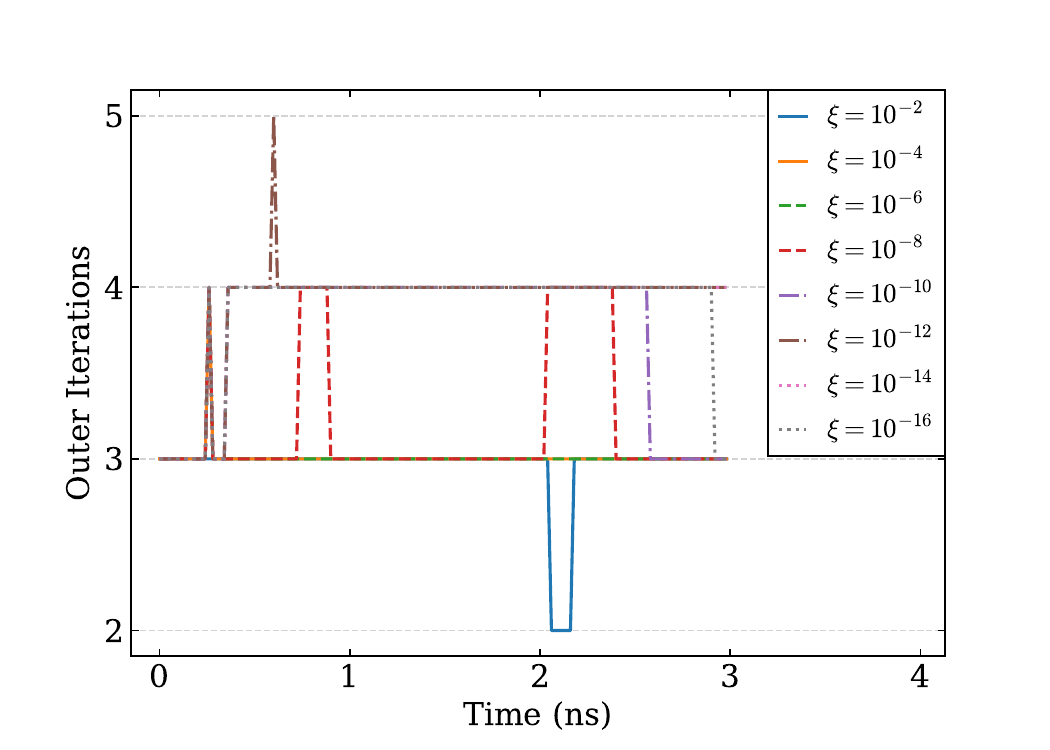}
	\caption{ Outer iteration count ($s$) per time step using the ROM
 (ref. Alg. \ref{alg:rom_alg}) with $s_\text{max}=\infty$ and various $\xi$ values.
Relative convergence criterion is set to $10^{-14}$.
}
	\label{fig:nbte_RTits}
\end{figure}

Fig. \ref{fig:nbte_RTits} plots  the  outer iteration counts $(s)$ for the ROM
using $s_\text{max}=\infty$ for different values of $\xi$ at each time step. Most time steps are solved in 3 or 4  outer
iterations using the ROM, with only a single time step requiring 5 iterations for several $\xi$. The iteration count increases as $\xi$ is decreased; for $\xi\geq 10^{-8}$ most time steps are solved with 3 outer iterations, whereas for $\xi\leq 10^{-10}$ most time steps require 4 solves.

%
\subsection{Error Analysis \& Convergence of ROMs}
The accuracy of the ROM w.r.t. the FOM solution in discrete space is now analyzed, along with the convergence behavior of its errors with $\xi$. Two cases are investigated: (i) $s_\text{max}=\infty$ and (ii) $s_\text{max}=1$ for both the offline and online stages of the ROM.
Let the ROM using $s_\text{max}=s_1$ outer iterations in the offline FOM calculations and $s_\text{max}=s_2$ for the online ROM iterations be written concisely as model $(s_1 | s_2)$. Therefore case (i) is model $(\infty | \infty)$ and case (ii) is model $(1 | 1)$. In each case the online stage of the ROM is consistent with the data generated in the offline stage, and therefore both models are expected to see convergence in their errors with $\xi$ to the FOM solution used in generating their POD bases.

\begin{figure}[ht!]
	\centering
	\subfloat[Material Temperature]{\includegraphics[width=.5\textwidth]{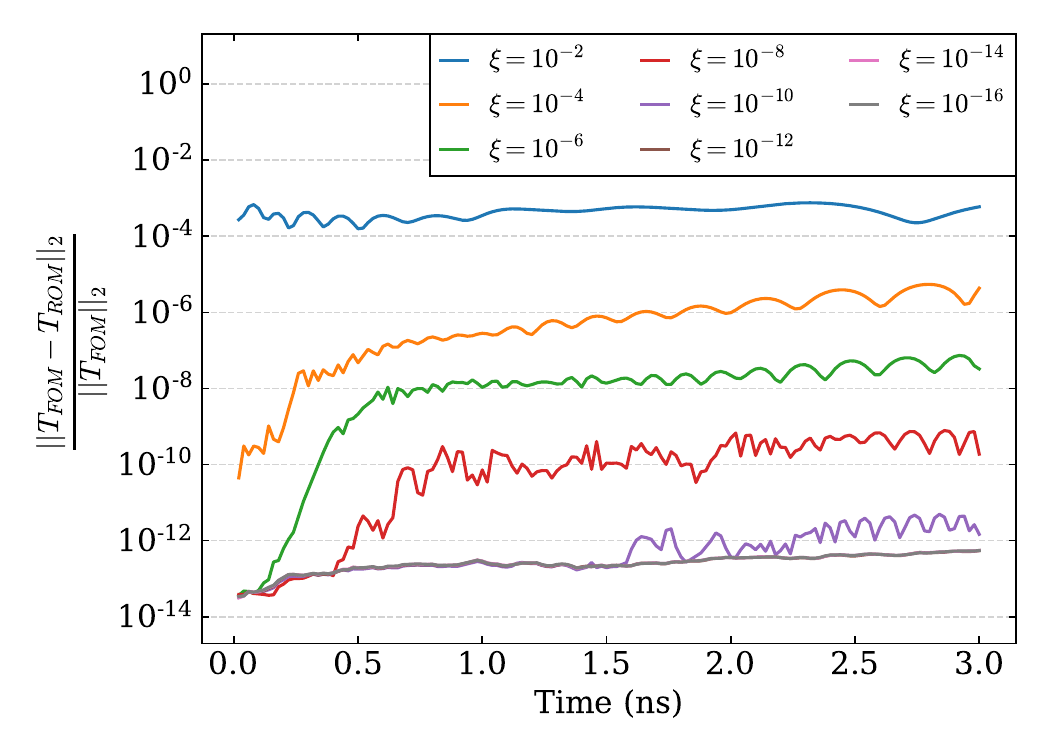}}
	\subfloat[Radiation Energy Density]{\includegraphics[width=.5\textwidth]{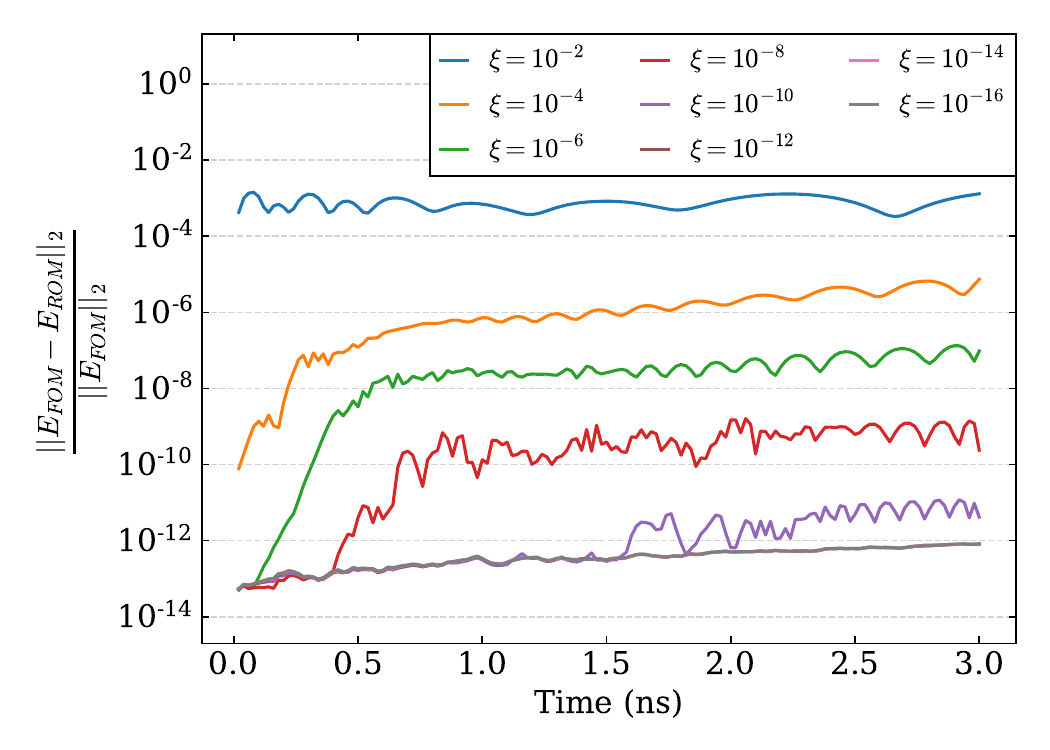}}
	\caption{Relative errors w.r.t. the FOM using $s_\text{max}=\infty$ in the 2-norm of the  ROM using $s_\text{max}=\infty$ in its offline and online stages (ref. Alg. \ref{alg:rom_alg}) for several $\xi$, plotted vs time}
	\label{fig:2nrm-errs_Phi-ref}
	\centering
	\subfloat[Material Temperature]{\includegraphics[width=.5\textwidth]{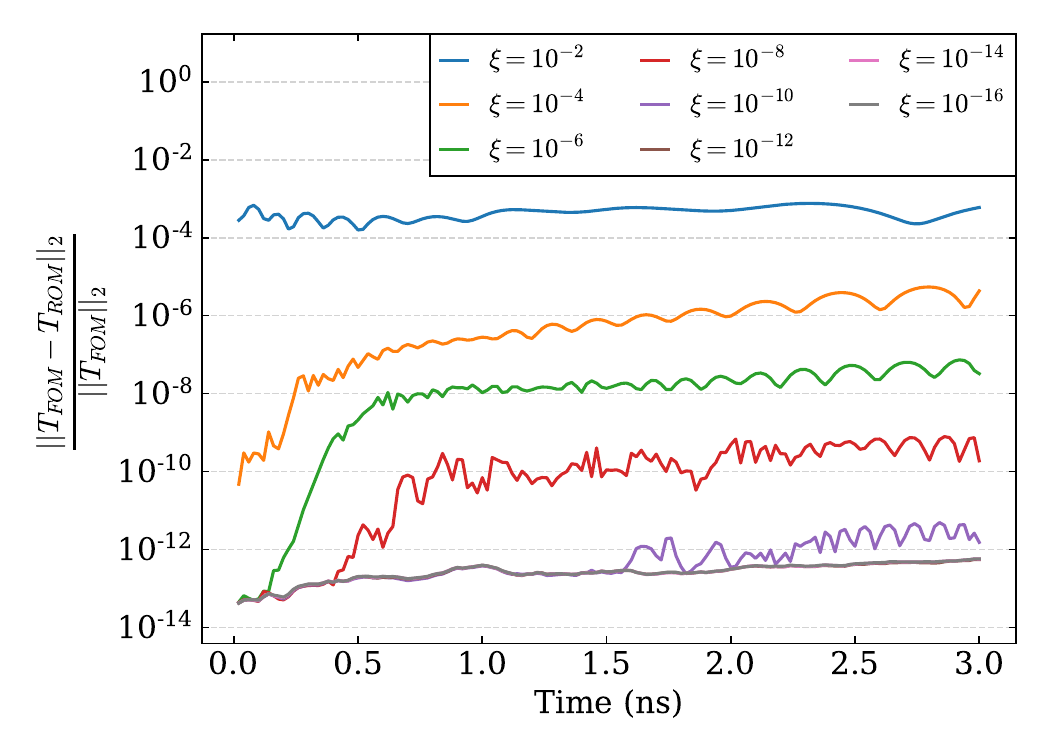}}
	\subfloat[Radiation Energy Density]{\includegraphics[width=.5\textwidth]{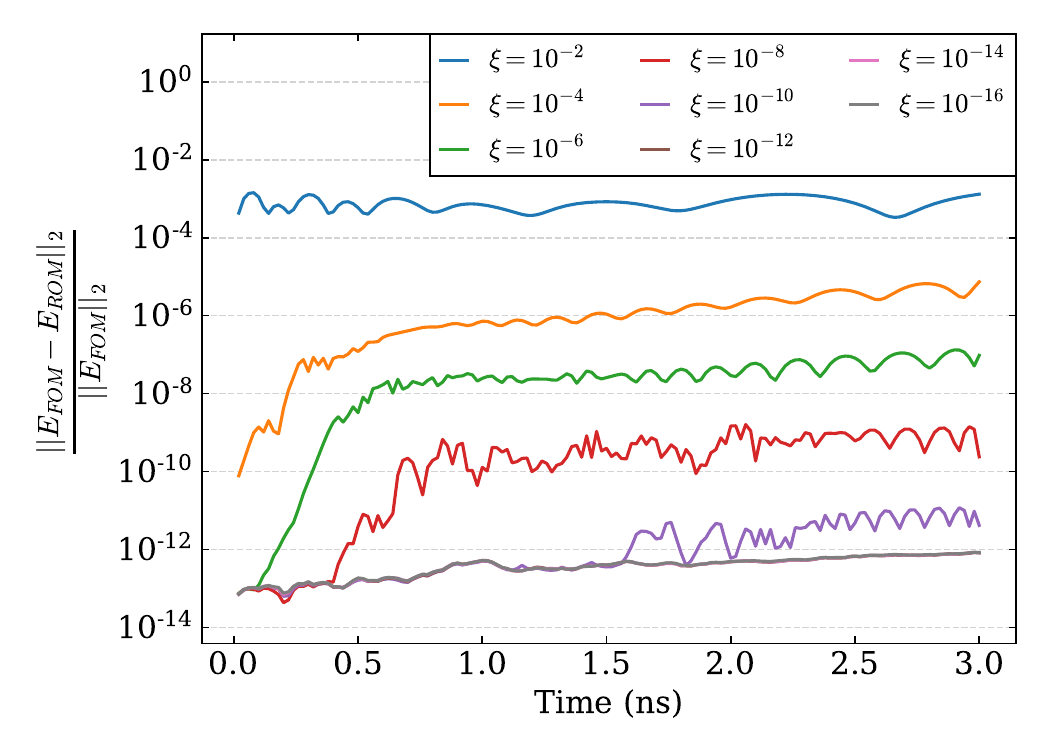}}
	\caption{Relative errors w.r.t. the FOM using $s_\text{max}=1$ in the 2-norm of the
 ROM using $s_\text{max}=1$ in its offline and online stages (ref. Alg. \ref{alg:rom_alg}) for several $\xi$, plotted vs time}
	\label{fig:2nrm-errs_Phi-ref_rt1_db-rt1_comp-rt1}
\end{figure}

Figs. \ref{fig:2nrm-errs_Phi-ref} and \ref{fig:2nrm-errs_Phi-ref_rt1_db-rt1_comp-rt1} plot the relative errors of the material temperature $(T)$ and total radiation energy density $(E)$ computed
with the ROM in the 2-norm over space at every instance of time for model $(\infty | \infty)$ and model $(1 | 1)$, respectively. The errors are calculated w.r.t. the FOM solution used to generate their respective databases (i.e. using $s_\text{max}=\infty$ and $s_\text{max}=1$, respectively). Each curve on the plots corresponds to a different value of $\xi$ used to define the ROM. As $\xi$ is decreased, the ROM errors decrease uniformly in time and stagnate after $\xi<10^{-10}$ at a level less than $10^{-12}$ which can be considered converged to the FOM solution within the bounds of finite precision. For all considered $\xi$ except for $\xi=10^{-2}$, the relative errors for both $T$ and $E$ increase rapidly in the beginning stage of the problem before leveling off, only slightly increasing with time.
Figs. \ref{fig:2nrm-errs_trend_Phi-ref} and \ref{fig:2nrm-errs_trend_Phi-ref_rt1_db-rt1_comp-rt1} plot the same relative errors for model $(\infty | \infty)$ and model $(1 | 1)$, respectively, vs $\xi$ instead of vs time. Each curve in these plots corresponds to a select instance of time. Here let the ROM errors be written as the function
$\vartheta(X) = \frac{\|X_\text{FOM}-X_\text{ROM}\|_2}{\|X_\text{FOM}\|_2}$ $(X=T,E)$.
Then for both models, the primary behavior is that 
$\vartheta(E)\approx\vartheta(T)\approx\xi\cdot10^{-2}$, until stagnation occurs on the level of
$\vartheta(E)\approx\vartheta(T)\approx 10^{-12}$.

\begin{figure}[ht!]
	\vspace{-.5cm}
	\centering
	\subfloat[Material Temperature]{\includegraphics[width=.5\textwidth]{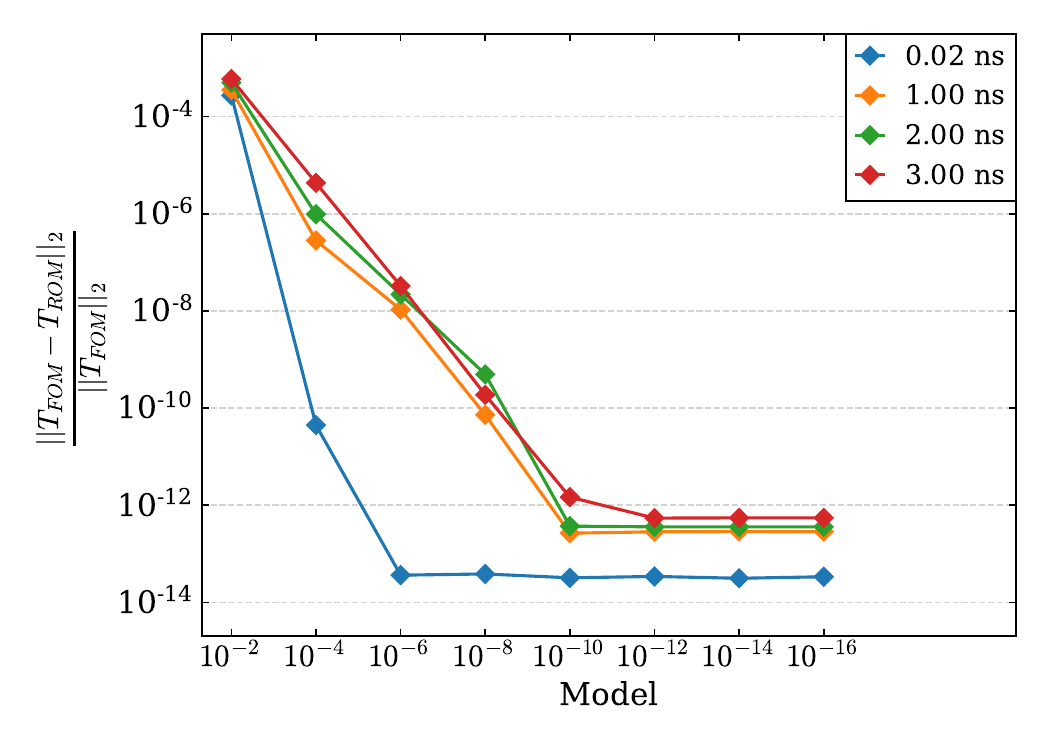}}
	\subfloat[Radiation Energy Density]{\includegraphics[width=.5\textwidth]{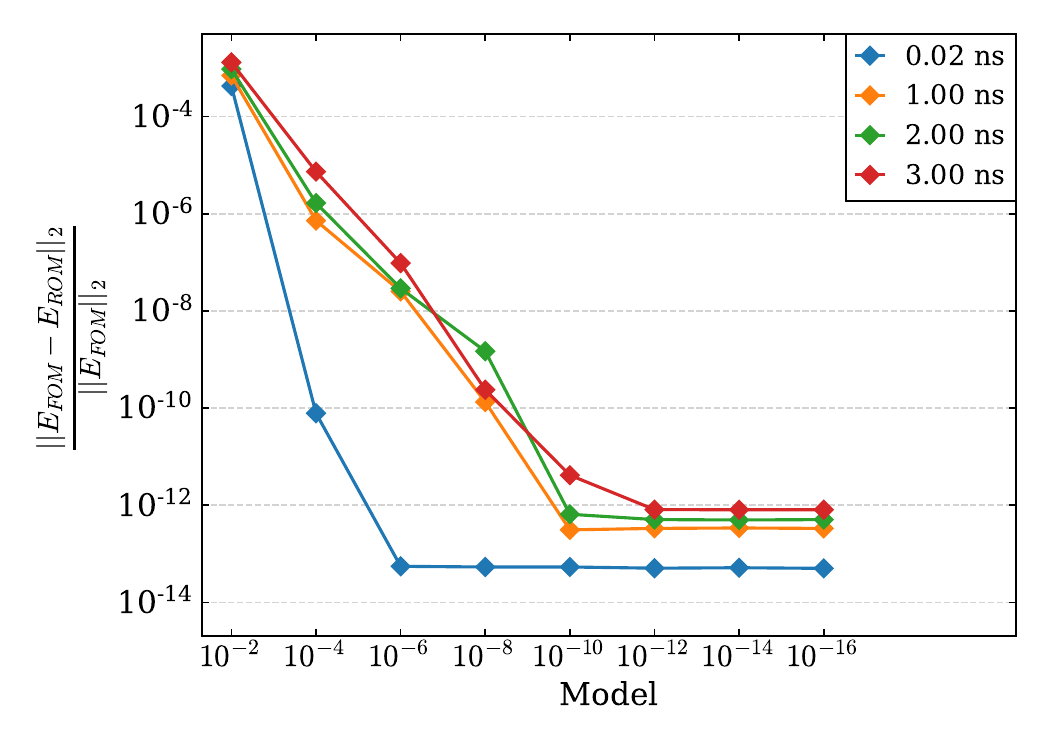}}
	\caption{Relative errors w.r.t. the FOM using $s_\text{max}=\infty$ in the 2-norm of the ROM using $s_\text{max}=\infty$ in its offline and online stages  (ref. Alg. \ref{alg:rom_alg}) at several times, plotted vs $\xi$}
	\label{fig:2nrm-errs_trend_Phi-ref}
	\centering
	\subfloat[Material Temperature]{\includegraphics[width=.5\textwidth]{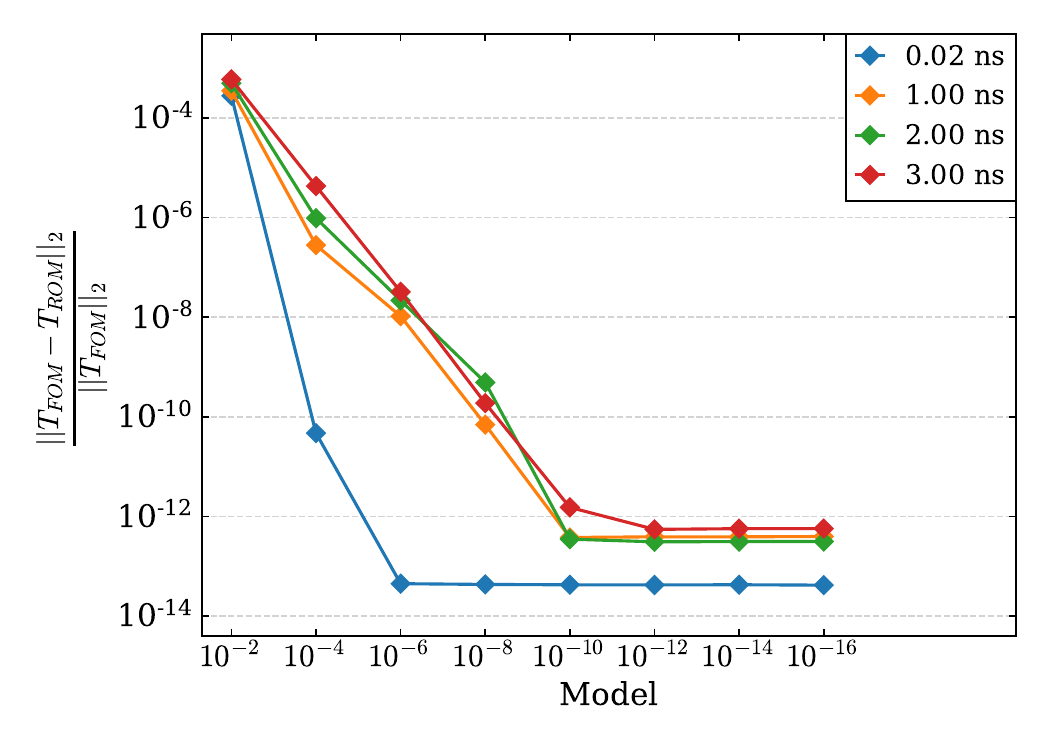}}
	\subfloat[Radiation Energy Density]{\includegraphics[width=.5\textwidth]{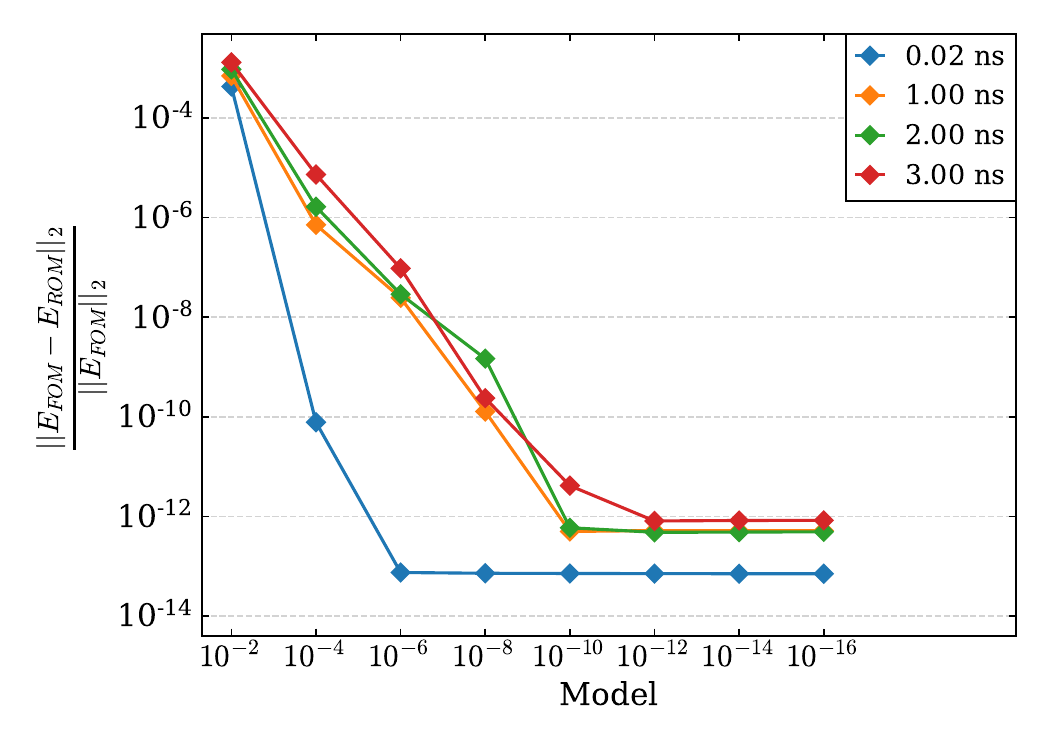}}
	\caption{Relative errors w.r.t. the FOM using $s_\text{max}=1$ in the 2-norm of the ROM using $s_\text{max}=1$ in its offline and online stages  (ref. Alg. \ref{alg:rom_alg}) at several times, plotted vs $\xi$}
	\label{fig:2nrm-errs_trend_Phi-ref_rt1_db-rt1_comp-rt1}
\end{figure}

%
\subsection{Performance of the ROM using a Single Outer Iteration}
In this section, the ROM using $s_\text{max}=1$ in its online stage will be evaluated w.r.t. the FOM solution found using $s_\text{max}=\infty$ to determine how well this
version of the ROM can reproduce the fully-converged F-C test solution. Cases when the POD basis functions are generated using $s_\text{max}=\infty$ and $s_\text{max}=1$ are considered, i.e. model $(\infty | 1)$ and model $(1 | 1)$. This is done to quantify effects on the ROM accuracy stemming from the quality of used POD basis functions.

Fig. \ref{fig:2nrm-errs_Phi-ref_rt1_db-fom_comp-fom} plots the relative errors of $T$ and $E$ in the 2-norm over space at every instance of time for model$(\infty | 1)$ w.r.t. the FOM solution generated with $s_\text{max}=\infty$. Fig. \ref{fig:2nrm-errs_Phi-ref_rt1_db-rt1_comp-fom} plots the same errors calculated for the model$(1 | 1)$ solution. The error levels in both $T$ and $E$ for all $\xi$ are nearly identical between the two models. Only $\xi\geq 10^{-8}$ is shown in both plots, as the error levels saturate by $\xi=10^{-8}$ and no longer decrease with further increases in rank. This saturation level in accuracy for both models is actually the same relative difference between the FOM solutions using $s_\text{max}=\infty$ and $s_\text{max}=1$. Note that $\xi=10^{-2}$ is the only value that provides a solution whose errors are above the saturation point for all time instances. When $\xi=10^{-4}$, the saturation point is met for times between 0 and 1.5 ns, where the largest errors are present. This is because as the problem evolves in time, the transients evolve more slowly and can be better captured by a scheme with a single outer iteration.

These results establish that when the ROM uses $s_\text{max}=1$ in its online stage, there is little accuracy lost when $s_\text{max}=1$ is also used in the offline stage to generate the ROM data, compared to using $s_\text{max}=\infty$ in the offline stage. The observed accuracy for this model is also at the level seen with the FOM solution generated using $s_\text{max}=1$, regardless of $s_\text{max}$ used in the offline stage.

\begin{figure}[hb!]
	\centering
	\subfloat[Material Temperature]{\includegraphics[width=.5\textwidth]{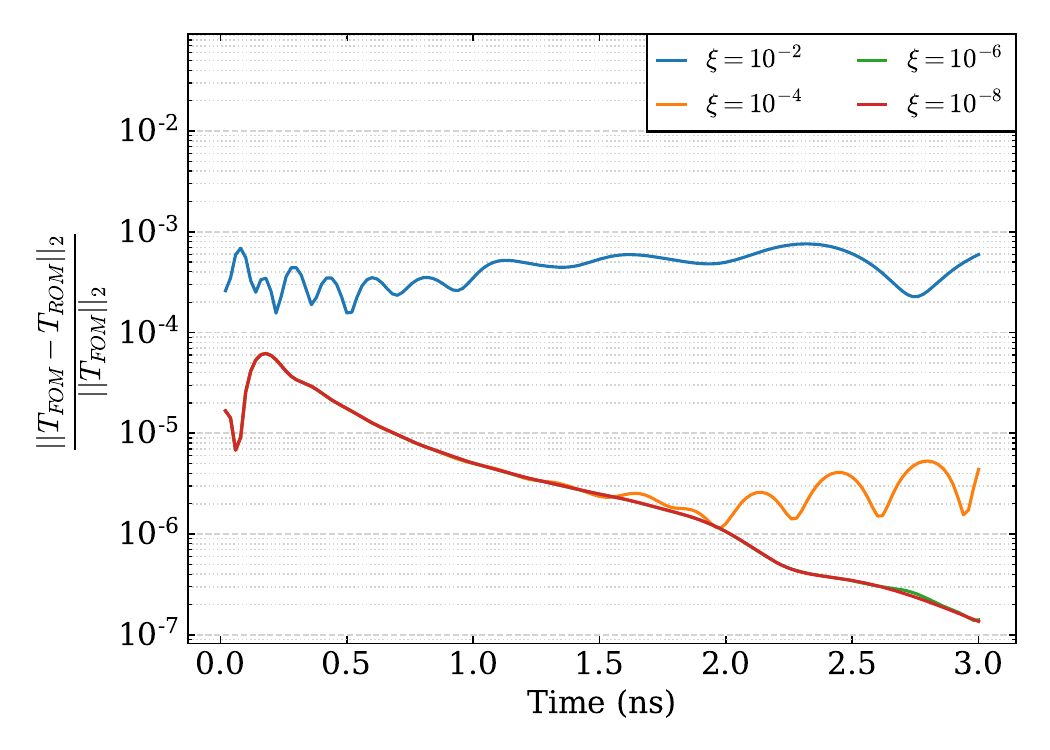}}
	\subfloat[Radiation Energy Density]{\includegraphics[width=.5\textwidth]{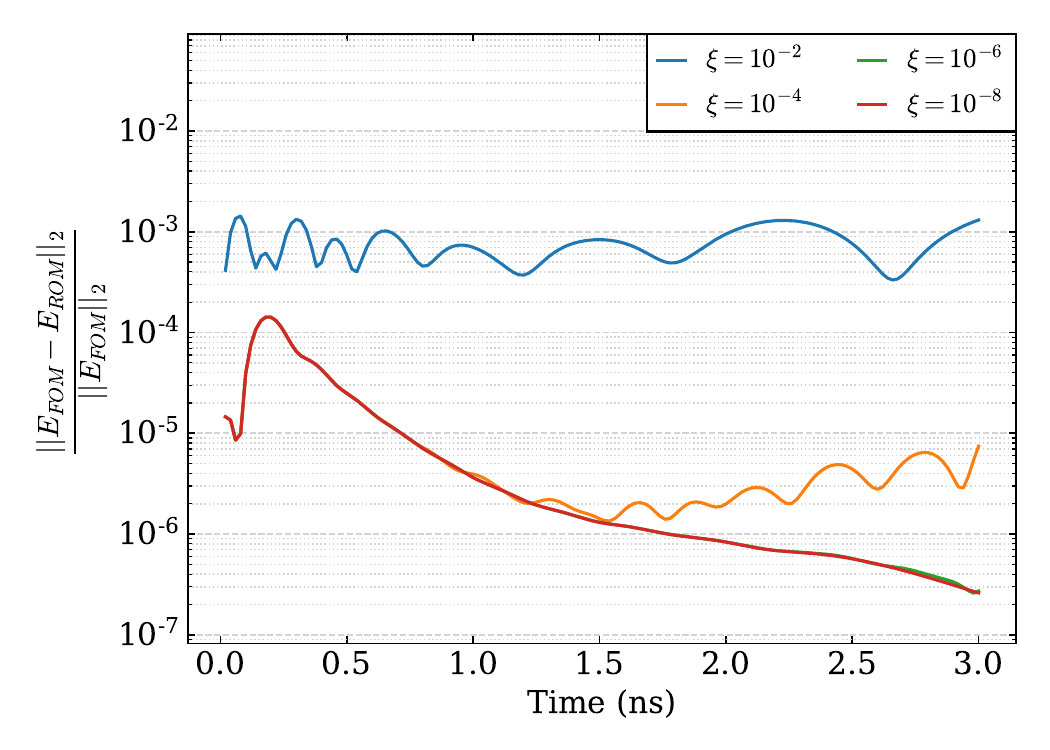}}
	\caption{Relative errors w.r.t. the FOM using $s_\text{max}=\infty$ (ref. Alg. \ref{alg:rom_alg}) in the 2-norm of the ROM using $s_\text{max}=1$ in its online stage (ref. Alg. \ref{alg:rom_alg}) and $s_\text{max}=\infty$ in its offline stage for several $\xi$, plotted vs time}
	\label{fig:2nrm-errs_Phi-ref_rt1_db-fom_comp-fom}
	\centering
	\subfloat[Material Temperature]{\includegraphics[width=.5\textwidth]{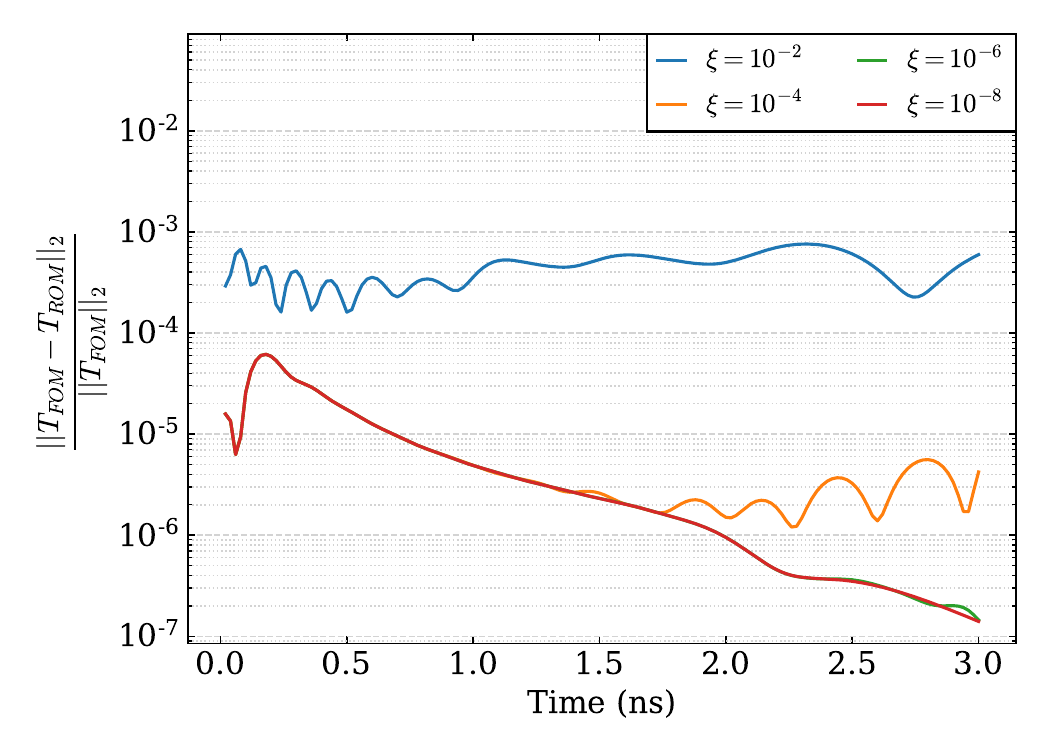}}
	\subfloat[Radiation Energy Density]{\includegraphics[width=.5\textwidth]{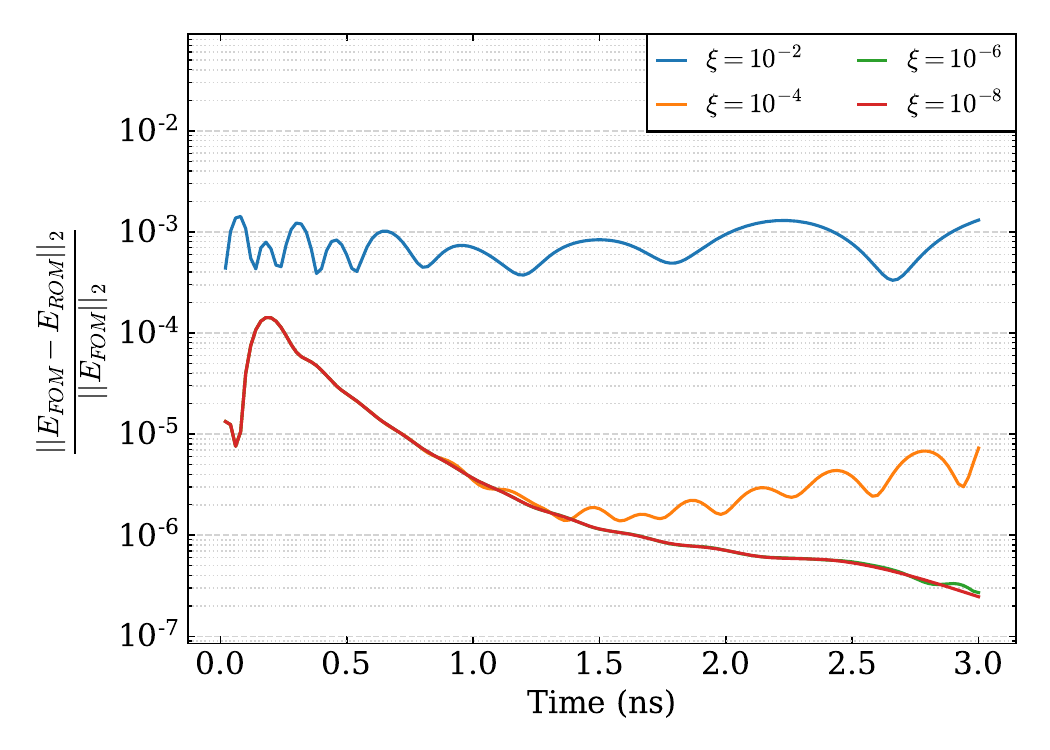}}
	\caption{Relative errors w.r.t. the FOM using $s_\text{max}=\infty$ (ref. Alg. \ref{alg:rom_alg}) in the 2-norm of the ROM using $s_\text{max}=1$ in both its online (ref. Alg. \ref{alg:rom_alg}) and offline stages for several $\xi$, plotted vs time}
	\label{fig:2nrm-errs_Phi-ref_rt1_db-rt1_comp-fom}
\end{figure}

\clearpage
%
\subsection{Analysis of Radiation Wave Propagation}
The next step in this analysis is to consider how well the proposed ROM captures
the essential
macroscopic
physics of the travelling wave.
The F-C test problem mimics the important class of high-temperature laser driven radiation shock experiments \cite{Moore-2015,guymer-2015,fryer-2016,fryer-2020,fryer-2022}.
One metric of particular importance in these studies is the {\it breakout time} of radiation \cite{Moore-2015,fryer-2016}.
In the context of ROM construction, breakout time is an integral quantity that can be used to quantify how well the ROM reproduces the nonlinear wavefront of radiation as it propagates across the problem domain.
The radiation wave produced in the F-C test travels from left to right and correspondingly the notion of breakout time is associated with radiation levels at the right boundary.
Typically in the literature, breakout time is measured as the elapsed time until a certain level of radiative flux is detected \cite{Moore-2015,fryer-2016}.
Here we consider not only the radiation flux, but the energy density and material temperature at the right boundary of the F-C test as well.
We consider the boundary-averages of these quantities in time, defined as follows:
\begin{equation}
 		\bar{F}_R = \frac{1}{L_R}\int_{0}^{L_R} \boldsymbol{e}_x\cdot\boldsymbol{F}(x_R,y)\ dy, \quad \bar{E}_R = \frac{1}{L_R}\int_{0}^{L_R} E(x_R,y)\ dy, \quad  \bar{T}_R = \frac{1}{L_R}\int_{0}^{L_R} T(x_R,y)\ dy,
 \end{equation}

\noindent where $L_R = x_R = 6\text{cm}$. The FOM solution for these quantities is shown in Fig. \ref{fig:rbndvals_fom}. Each of $\bar{F}_R$, $\bar{E}_R$ and $\bar{T}_R$ represent a different measure of the amount of radiation that has reached the drive-opposite side of the test domain over time, which is what might be measured by a detector during an experiment.
\begin{figure}[ht!]
	\centering
	\subfloat[ $\bar{F}_R$]{\includegraphics[width=.33\textwidth]{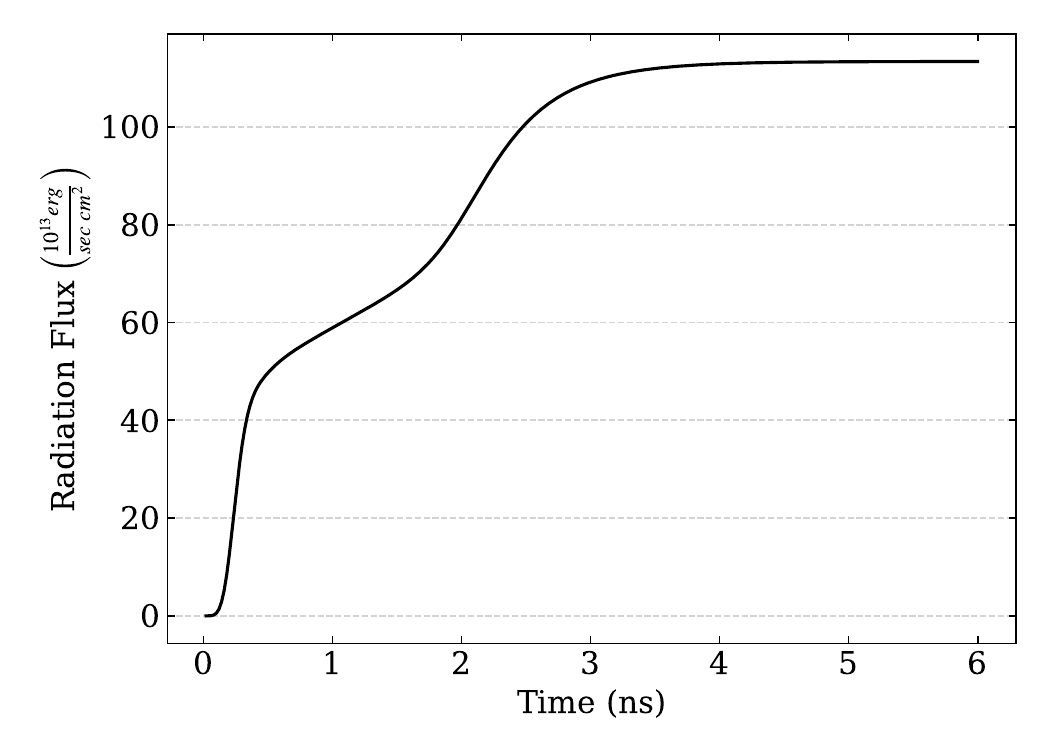}}
	\subfloat[$\bar{E}_R$]{\includegraphics[width=.33\textwidth]{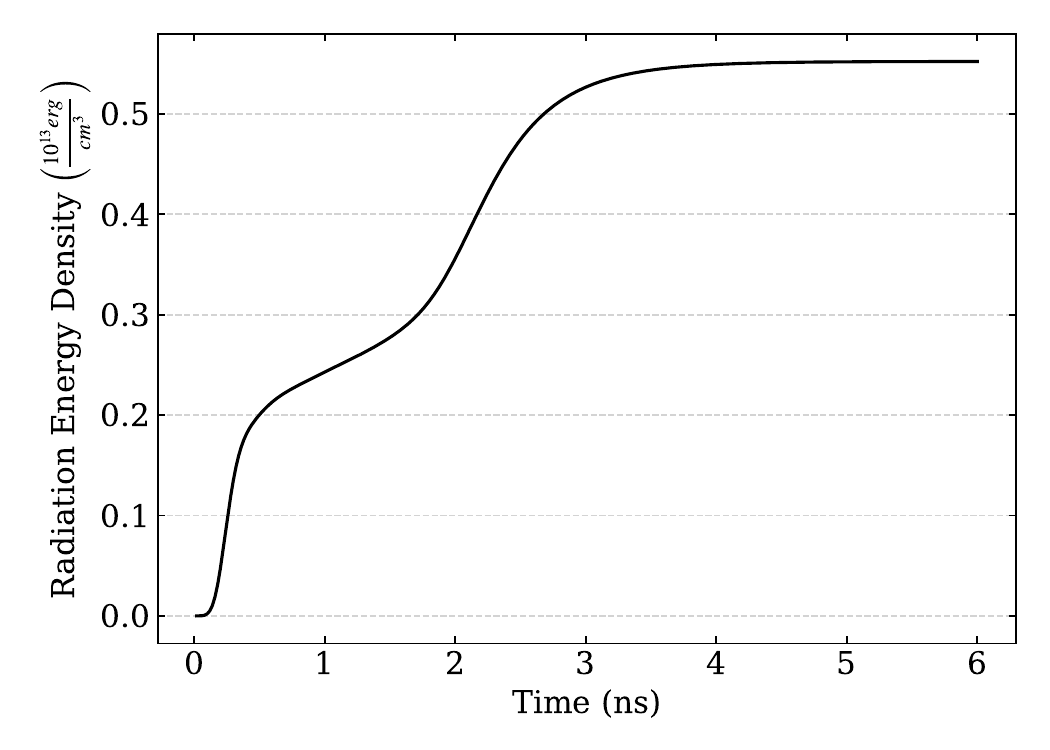}}
	\subfloat[$\bar{T}_R$]{\includegraphics[width=.33\textwidth]{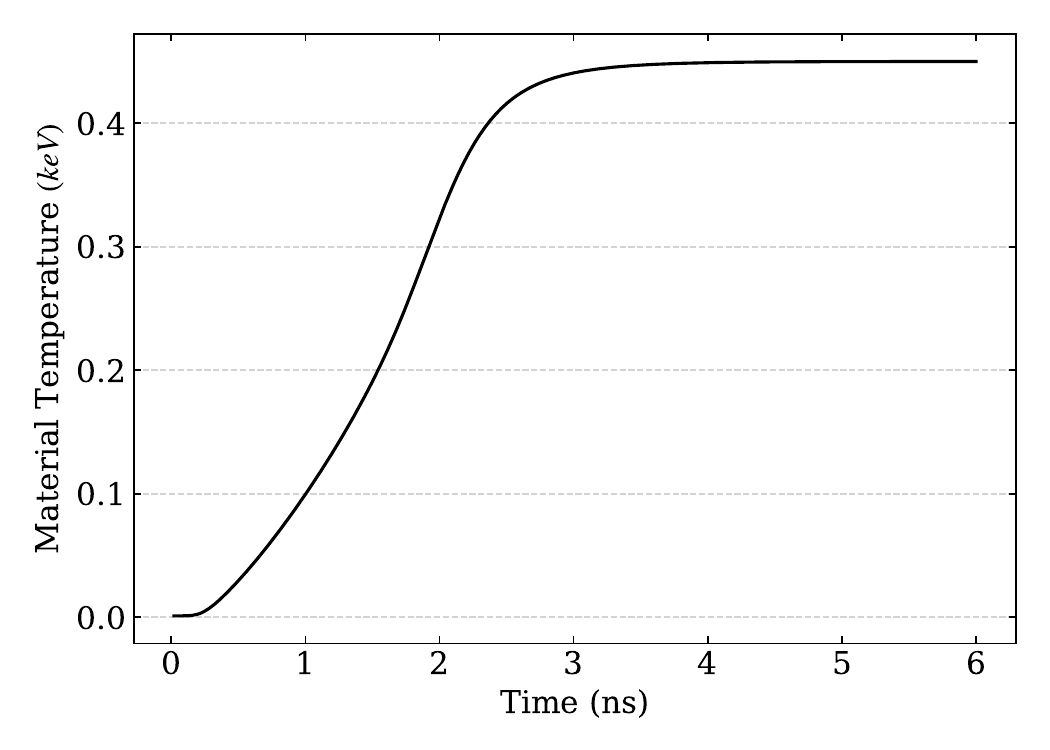}}
	\caption{Total radiation flux $(\bar{F}_R)$, total radiation energy density $(\bar E_R)$ and material temperature $(\bar T_R)$
		averaged over the right boundary of the spatial domain plotted vs time. Shown solutions are generated by the FOM.
		\label{fig:rbndvals_fom} }
\end{figure}

Fig. \ref{fig:rbndvals} plots the relative error for each quantity produced by model $(\infty | \infty)$ at each instance of time w.r.t. the FOM solution. All considered values of $\xi$ are included here, and similar behavior to Fig. \ref{fig:2nrm-errs_Phi-ref} is seen. Each quantity is observed to converge to the FOM solution as $\xi$ decreases, and errors propagate in time with a slight rate of increase for $\xi\leq 10^{-4}$. The errors in $\bar{F}_R$, $\bar{E}_R$ and $\bar{T}_R$ are generally higher than the 2-norm error of $T$ and $E$ for any given $\xi$ before the limits of numerical precision become dominant. This difference in error levels is between one and two orders of magnitude. The observed errors are all approximately bounded by $\xi$.

\begin{figure}[ht!]
	\centering
	\subfloat[$\bar{F}_R$]{\includegraphics[width=.5\textwidth]{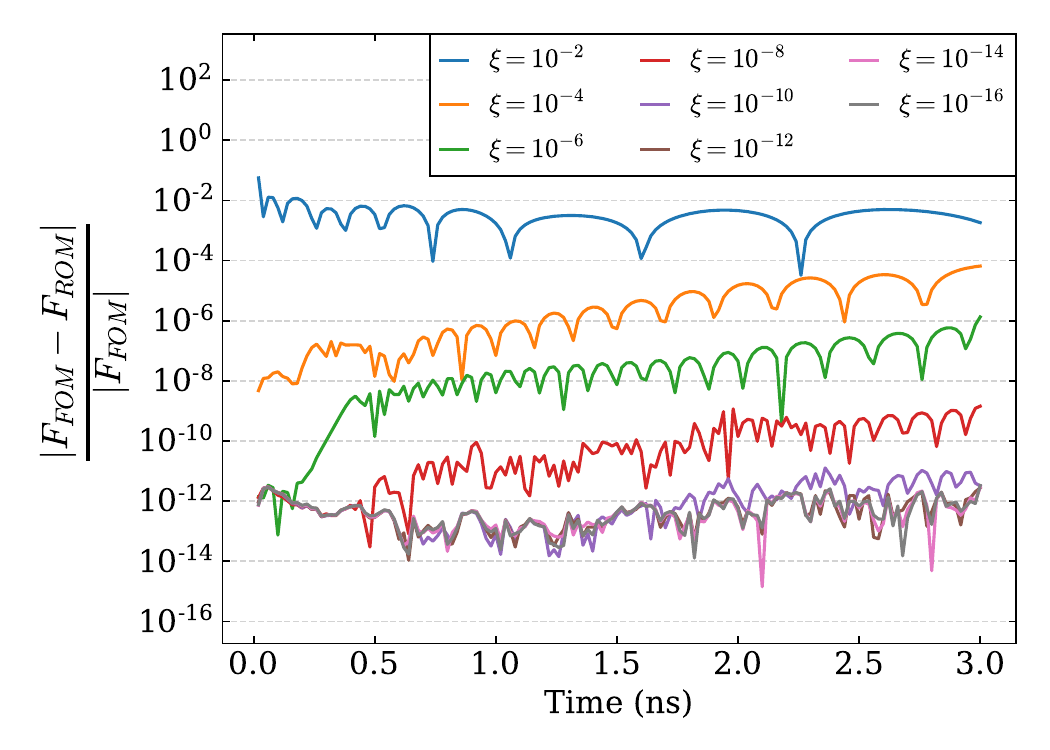}}\\
	\subfloat[$\bar E_R$]{\includegraphics[width=.5\textwidth]{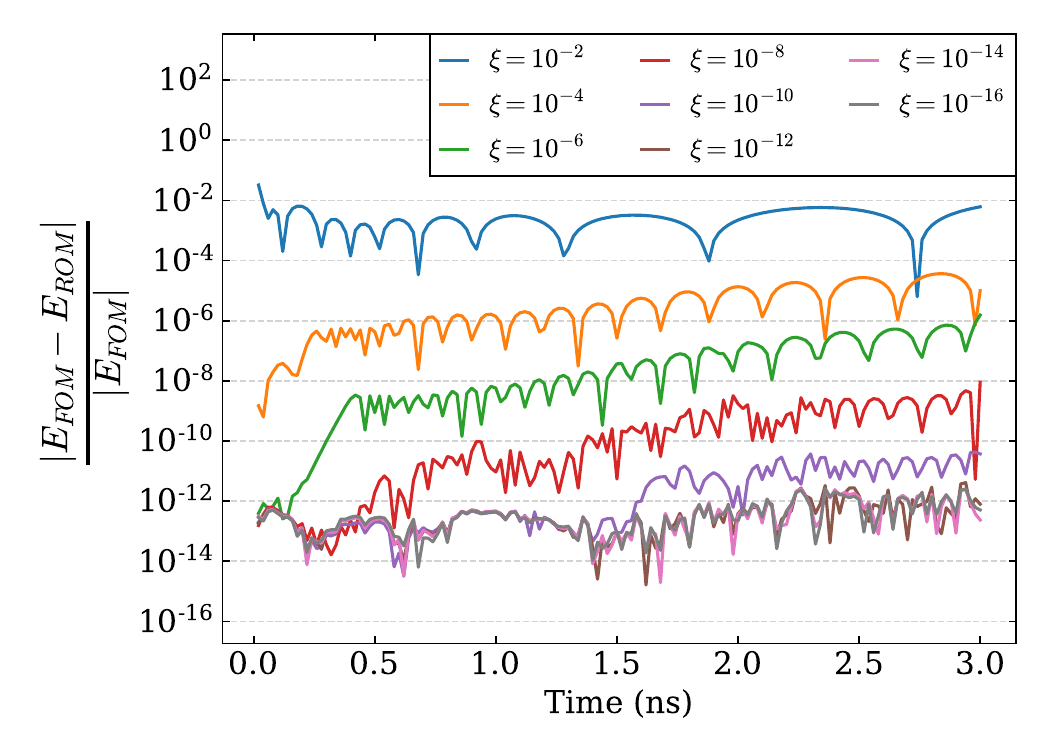}}
	\subfloat[$\bar T_R$]{\includegraphics[width=.5\textwidth]{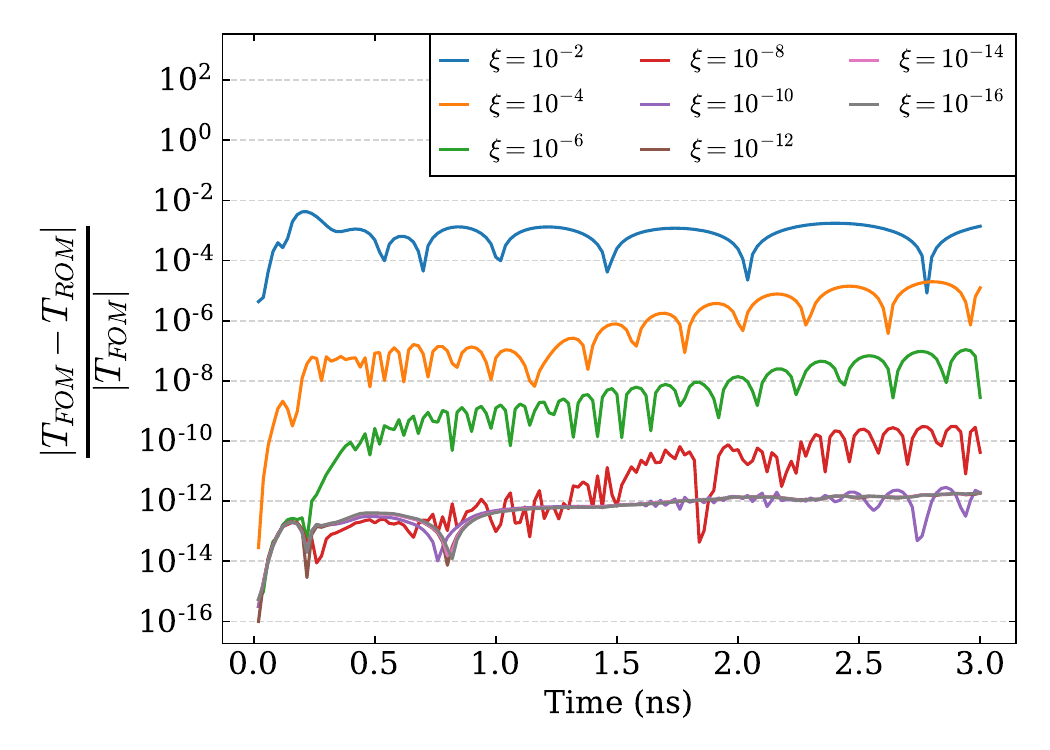}}
	\caption{Relative error for the ROMs using $s_\text{max}=\infty$ in the offline and online stages with various $\xi$ values
		for data located at and averaged over the right boundary of the domain.
		\label{fig:rbndvals} }
\end{figure}

%
Another important metric is the spectrum of the radiation wavefront. We analyze the radiation spectrum in the F-C test by calculating errors in the ROM radiation energy densities. These energy densities are averaged over the frequency interval they represent to produce $\bar{E}_g=\frac{E_g}{\nu_{g} - \nu_{g-1}}$.
We analyze the energy spectrum at the bottom corner of the right boundary of the F-C test in space. The solution at this point will be highly anisotropic.

 Fig. \ref{fig:Eg_spectrum} plots the frequency spectrum of radiation energy densities $(\bar{E}_g)$ for the F-C test obtained by the FOM at the chosen point in space.
This plot graphs the radiation energy densities vs photon energy
 at several instants of time (each curve represents a different instance of time).
 Each point is located at the center of a discrete energy group (see Table \ref{tab:freq_grps}) on the frequency-axis. The instants of time were chosen based off of Fig. \ref{fig:rbndvals_fom} to sufficiently sample temporal regions of interest.
The energy range is selected so that all energy groups except the last are shown.

\begin{figure}[ht!]
	\centering
	\includegraphics[width=.5\textwidth]{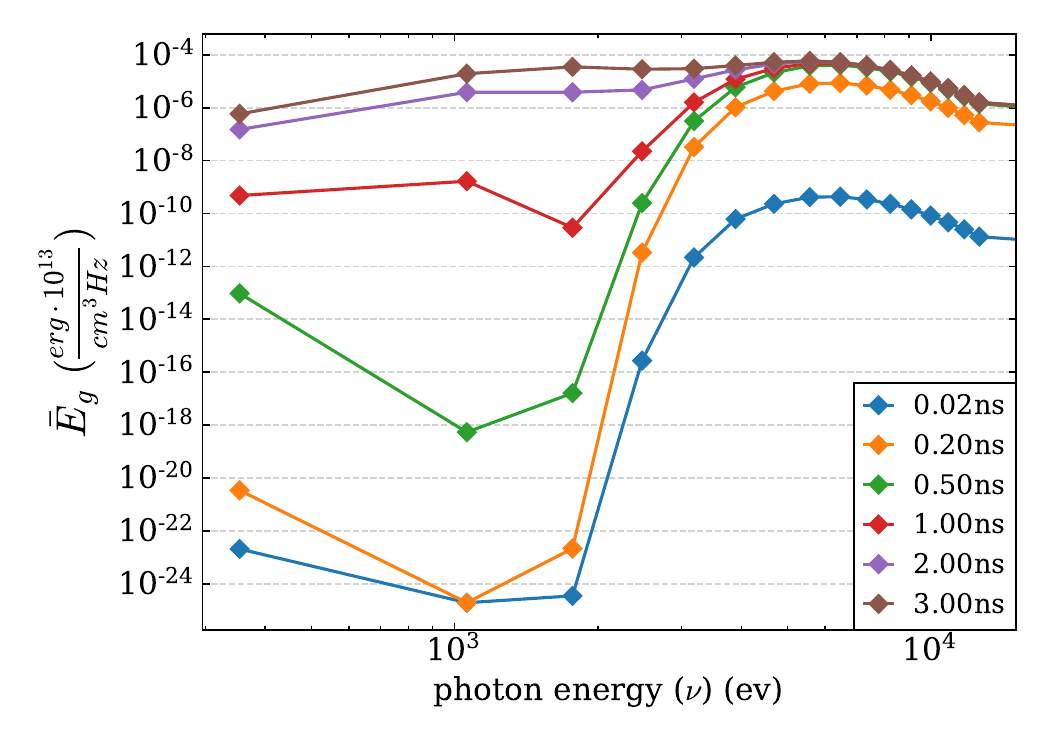}
	\caption{Radiation energy density spectrum located at the bottom corner of the right boundary, taken at several time instances.
		\label{fig:Eg_spectrum} }
\end{figure}

Fig. \ref{fig:Eg_spectrum_err_norms_crn} plots the relative errors of the ROM solution for $\bar{E}_g$ using model $(\infty | \infty)$ in the temporal 2-norm defined by
\begin{equation}
\| x(t) \|_2^t=\bigg(\int_{0}^{t^\text{end}}x(t)^2dt\bigg)^{1/2}.
\end{equation}
All considered values of $\xi$ are shown here to give a complete picture of how the radiation spectrum
converges in a temporal-integral sense. For all $\xi\leq 10^{-4}$ and before numerical precision becomes limiting, the high frequency groups are reproduced with significantly higher accuracy than low frequency groups. Each energy group decreases in error when $\xi$ is decreased.
\begin{figure}[ht!]
	\centering
	\includegraphics[width=.5\textwidth]{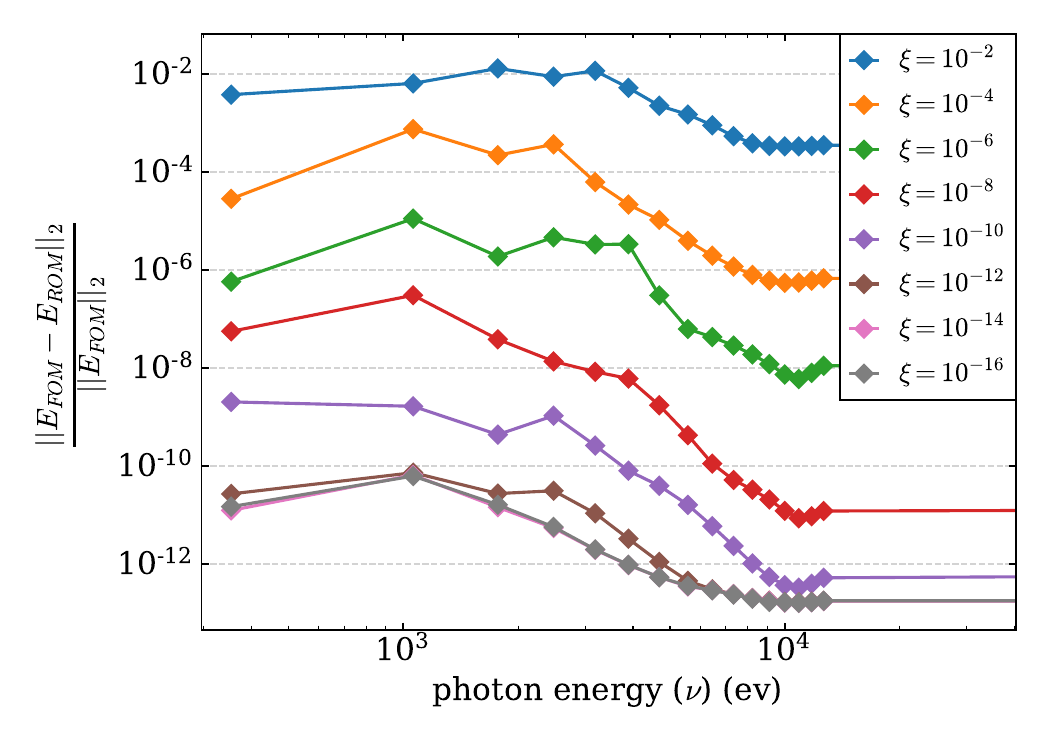}
	\caption{Relative errors of the radiation spectrum produced by the ROM using $s_\text{max}=\infty$ located at the bottom corner of the right boundary in the temporal 2-norm.
		\label{fig:Eg_spectrum_err_norms_crn} }
\end{figure}

%
%
\clearpage
\section{Conclusion} \label{sec:conclusion}
In this paper, a new ROM for TRT problems is presented based on (i) a POD-Petrov-Galerkin projection of the NBTE with scaling by material opacities and (ii) moment equations of the BTE.
A multilevel set of moment equations of the BTE is derived by means of the nonlinear projective approach.
The Eddington tensor which formulates exact closure for the system of moment equations is defined as a linear function of the NBTE solution.
A low-dimensional POD-Petrov-Galerkin projection of the NBTE over the entire phase space of the problem is constructed for the proposed ROM.
The Eddington tensor is expanded directly in a superposition of the proper orthogonal modes which represent the NBTE solution.
The solution of the projected NBTE defines coefficients of this expansion.

Detailed analysis of the ROM's properties and behavior was performed on the 2D F-C test problem which simulates evolution of a high-temperature supersonic radiation wavefront as it travels through a medium.
The ROM was shown to produce solutions to the F-C test with sufficient accuracy for practical simulations with very low-rank representations of the NBTE.
When outer iterations
are allowed to fully converge in both the offline and online stages, the ROM solution converges to the FOM solution with predicable behavior.
When only a single outer iteration
is allowed per time step during the offline and online stages, the ROM is able to reproduce the fully-converged FOM solution with
sufficiently high accuracy. The ROM was also demonstrated to reproduce essential physical characteristics of the radiation wavefront in the F-C test, including (i) the amount of radiation present on the right-hand side of the domain throughout the time domain of the problem, and (ii) the spectral makeup of the radiation.

The ROM has been shown capable of producing accurate solutions to the F-C test with low-rank expansions, and future development is promising.
The ROM is a good candidate for continued development into the use of parameterization and hyper-reduction.
A potential parameterization for this ROM would be to interpolate between basis functions for different realizations of the F-C test using interpolants developed with a dimensionless study of the problem \cite{jc-dya-jqsrt-2023}.

\section{Acknowledgements}
Los Alamos Report LA-UR-23-29764.
The project or effort depicted is sponsored  by the Department of Defense, Defense Threat Reduction Agency, grant number HDTRA1-18-1-0042.
This work was supported by the U.S. Department of Energy through the Los Alamos National Laboratory. Los Alamos National Laboratory is operated by Triad National Security, LLC, for the National Nuclear Security Administration of U.S. Department of Energy (Contract No. 89233218CNA000001).
The content of the information does not necessarily reflect the position or the policy of the federal government, and no official endorsement should be inferred.

%
%
\appendix
\section{Derivation of Test Basis Functions} \label{apdx:test_basis}

In this appendix, a proof is provided via Theorem \ref{thm:ch3:pg_opt} to show the solution to Eq. \eqref{eq:nbte_R_proj} satisfies the condition \eqref{eq:pg_min} when test basis functions are defined by Eq. \eqref{eq:pg_basis}.
We use a weight matrix $\mathbf{W}$ (Eq. \eqref{eq:Wmat}) which is by definition symmetric positive definite (SPD).
This allows for use of Lemma \ref{lem:ch3:w_conv}.

\begin{lemma} \label{lem:ch3:w_conv}
	The norm $\|\cdot\|_W$ is strictly convex if $\mat{W}$ is symmetric positive definite
\end{lemma}
\begin{proof}
	By the definition of the $W$ norm,
	\begin{align}
		\| \boldsymbol{x} \|_W^2 &= \boldsymbol{x}^\top\mat{W}\boldsymbol{x}.
	\end{align}
	
	\noindent Let $f(\boldsymbol{x}) = \| \boldsymbol{x} \|_W^2$, then
	\begin{equation}
		\frac{d f(\boldsymbol{x})}{d \boldsymbol{x}} = 2\boldsymbol{x}^\top\mat{W},\quad \frac{d^2 f(\boldsymbol{x})}{d \boldsymbol{x}^2} = 2\mat{W},
	\end{equation}
	
	\noindent since $\mat{W}$ is symmetric. The Hessian of $f(\boldsymbol{x})$ is therefore positive definite, since $\mat{W}$ is positive definite. This implies $f(\boldsymbol{x})$ strictly convex.
\end{proof}

\begin{theorem} \label{thm:ch3:pg_opt}
	The solution to Eq. \eqref{eq:nbte_R_proj} with test basis functions $\boldsymbol{\psi}_\ell = \frac{d \mathcal{R}^n(\lambda_\ell\boldsymbol{u}_\ell)}{d \lambda_\ell}$ satisfies the following minimization problem
	\begin{equation} \label{eq:ch3:wmin}
		\boldsymbol{\Lambda} = \arg\min_{\boldsymbol{\zeta}\in\real^k}\| \mathcal{R}^n(\mat{U}\boldsymbol{\zeta}) \|_W^2,
	\end{equation}
	
	\noindent where $\boldsymbol{\Lambda} = (\lambda_1 \dots \lambda_k)^\top$.
\end{theorem}
\begin{proof}
	Let $f(\boldsymbol{x}(\boldsymbol{\Lambda})) = \|\boldsymbol{x}(\boldsymbol{\Lambda})\|_W^2$ with $\boldsymbol{x}(\boldsymbol{\Lambda}) = \mathcal{R}^n(\mat{U}\boldsymbol{\Lambda})$. Assuming $\mat{W}$ symmetric positive definite, $f$ must have a unique minimizer $\hat{\boldsymbol{x}}=\boldsymbol{x}(\hat{\boldsymbol{\Lambda}})$ by Lemma \ref{lem:ch3:w_conv}. This point occurs at $\frac{df}{d\boldsymbol{x}}=0$, or equivalently at $\frac{df}{d\boldsymbol{\Lambda}}=0$ since $\frac{df}{d\boldsymbol{\Lambda}}=\frac{df}{d\boldsymbol{x}}\frac{d\boldsymbol{x}}{d\boldsymbol{\Lambda}}$. These latter two derivatives are
	\begin{equation}
		\frac{df}{d\boldsymbol{x}} = 2\boldsymbol{x}^\top\mat{W},\quad \frac{d\boldsymbol{x}}{d\boldsymbol{\Lambda}} = \frac{d\mathcal{R}^n(\mat{U}\boldsymbol{\Lambda})}{d\boldsymbol{\Lambda}}.
	\end{equation}
	
	\noindent Setting $\frac{df}{d\boldsymbol{\Lambda}}=0$ therefore leads to the following
	\begin{equation}
		\frac{df}{d\boldsymbol{x}}\frac{d\boldsymbol{x}}{d\boldsymbol{\Lambda}} = 2\boldsymbol{x}^\top\mat{W}\frac{d\mathcal{R}^n(\mat{U}\boldsymbol{\Lambda})}{d\boldsymbol{\Lambda}} = \boldsymbol{0}^\top,
	\end{equation}
	
	\noindent or equivalently
	\begin{equation} \label{eq:ch3:drdl_eq}
		\bigg( \frac{d\mathcal{R}^n(\mat{U}\boldsymbol{\Lambda})}{d\boldsymbol{\Lambda}} \bigg)^\top\mat{W}\boldsymbol{x} = \boldsymbol{0},
	\end{equation}
	
	\noindent where $\boldsymbol{0}\in\real^k$ is the vector of all 0's. Substituting the definition for $\boldsymbol{x}$ and letting $\mat{\Psi}=\frac{d\mathcal{R}^n(\mat{U}\boldsymbol{\Lambda})}{d\boldsymbol{\Lambda}}$
	in Eq. \ref{eq:ch3:drdl_eq} gives us the relation $\mat{\Psi}^T\mat{W}\mathcal{R}^n(\mat{U}\boldsymbol{\Lambda}) = 0$. Therefore the solution $\boldsymbol{\Lambda}$ to Eq. \eqref{eq:nbte_R_proj} with $\mat{\Psi}=\frac{d\mathcal{R}^n(\mat{U}\boldsymbol{\Lambda})}{d\boldsymbol{\Lambda}}$ satisfies the minimization problem in Eq. \ref{eq:ch3:wmin}.
\end{proof}

\bibliographystyle{model1-num-names}
\bibliography{jmc-dya-nbte-podpg-Arxiv}

\end{document}